\documentclass[11pt,draft]{amsart}
\usepackage[latin1]{inputenc}
\usepackage{a4wide}
\usepackage{setspace}
\usepackage{amssymb}

\numberwithin{equation}{section}

\newtheorem{lemma}{Lemma}[section]

\newtheorem{theorem}[lemma]{Theorem}
\newtheorem{proposition}[lemma]{Proposition}
\newtheorem{definition}[lemma]{Definition}
\newtheorem{corollary}[lemma]{Corollary}
\newtheorem{example}[lemma]{Example}
\newtheorem{exercise}[lemma]{Exercise}
\newtheorem{remark}[lemma]{Remark}
\newtheorem{fig}[lemma]{Figure}
\newtheorem{tab}[lemma]{Table}

\newcommand{\bth}{\begin{theorem}}
\newcommand{\ethe}{\end{theorem}}

\newcommand{\bre}{\begin{remark}\em }
\newcommand{\ere}{\end{remark}\noindent}

\newcommand{\ble}{\begin{lemma}}
\newcommand{\ele}{\end{lemma}\noindent}

\newcommand{\bde}{\begin{definition}}
\newcommand{\ede}{\end{definition}\noindent}

\newcommand{\bco}{\begin{corollary}}
\newcommand{\eco}{\end{corollary}\noindent}

\newcommand{\bpr}{\begin{proposition}}
\newcommand{\epr}{\end{proposition}\noindent}

\newcommand{\bexer}{\begin{exercise}}
\newcommand{\eexer}{\end{exercise}}

\newcommand{\bexam}{\begin{example}}
\newcommand{\eexam}{\end{example}}

\newcommand{\bfi}{\begin{fig}}
\newcommand{\efi}{\end{fig}}

\newcommand{\btab}{\begin{tab}}
\newcommand{\etab}{\end{tab}}

\newcommand{\beao}{\begin{eqnarray*}}
\newcommand{\eeao}{\end{eqnarray*}\noindent}

\newcommand{\beam}{\begin{eqnarray}}
\newcommand{\eeam}{\end{eqnarray}\noindent}

\newcommand{\bce}{\begin{center}}
\newcommand{\ece}{\end{center}}

\newcommand{\barr}{\begin{array}}
\newcommand{\earr}{\end{array}}

\newcommand{\bdis}{\begin{displaymath}}
\newcommand{\edis}{\end{displaymath}\noindent}

\newcommand{\ep}{\epsilon}

\newcommand{\la}{\lambda}

\newcommand{\eid}{\buildrel{\rm d}\over {=}}
\newcommand{\drt}{{D\sqrt{1-t}}}
\newcommand{\gt}{\Gamma_{\tau}}

\begin{document}

\title[Extremes of L\'evy processes]
{Extremes of L\'evy processes with light tails}
\author[M. Braverman]{Michael Braverman}
\address{Department of Mathematics\\ 
Ben-Gurion University of the Negev\\
P.O. Box 653 \\
Beer-Sheva 84105\\ 
Israel}
\email{braver@math.bgu.ac.il}
\thanks{Research supported by Kamea programm}

\keywords{ Poisson process, Brownian motion,
   supremum\vspace{.5ex}}

\begin{abstract}
Let $X(t)\, t\ge0\,, X(0)=0$, be a L\'evy process with spectral L\'evy measure $\rho$.
Assuming that $\rho((-\infty,0))<\infty$ and the right tail of $\rho$ is light, we show that in
the presence of Brownian component
$$
P\left(\sup_{0\le t \le 1}X(t)>u \right) \sim P\left(X(1)>u \right)\
$$
as $u\to\infty$. In the absence of Brownian component  these tails are not always comparable.  
An example of L\'evy process of the type $X(t)=B(t)+Z(t)$, where $B(t)$ is a Brownian motion
and $Z(t)$ is a compound Poisson process with positive jumps, 
for which these tails are incomparable is also given.
\end{abstract}

\maketitle

\section{Introduction}
\label{s:int}

The problem of finding asymptotics of the probabilities $P(\sup_{t\in T}X(t)>u)$ as $u\to\infty$, where 
$X(t)$ is a stochastic process, is a classical one. It was intensively studied, but  many
unsolved questions still remain. 

In what follows $T=[0,1]$ and $X(t)$ is a L\'evy process, $X(0)=0$. Its characteristic function is given by the  well
known L\'evy--Khintchin formula
$$
E\exp\left(isX(t)\right) = \exp\left(t\psi(s)\right)\,,
$$
where
\beam
\label{L.exp}
\psi(s) = -ibt -\frac{\sigma^2s^2}2 +\int_{-\infty}^\infty\left(e^{isx}-1 - isx{\mathbf 1}(|x|\le1)\right)\rho(dx)\,.
\eeam
Here $b\in {\mathbf R}, \sigma\ge0$ and $\rho$ is a Borel measure such that 
$\int_{-\infty}^\infty \min\{1, x^2\}\rho(dx)<\infty$ (the L\'evy measure). 

If $\sigma$ is strictly
positive, then the process can be represent as a sum of independent Brownian motion $B(t)$ and
another L\'evy process $X_1(t)$. In the case $\rho(\mathbf R) <\infty$ the last process
is a compound Poisson. So, if the  L\'evy measure is finite, we can write
\beam
\label{def.X}
X(t) = \sigma B(t)+Z(t) -bt,\,\,t\ge 0,
\eeam
where $Z(t)$ is a compound Poisson process with the parameter $\la=\rho(\mathbf R)$. It means that   
\beam
\label{comp.P}
Z(t) = \sum_{k=1}^{N(t)}X_k \,,
\eeam
where  $N(t)$ is a Poisson process with parameter $\la$ independent of iid random
variables $\{X_k\}_{k=1}^\infty$  (the {\it jumps} of the process).  

  One of the approaches to the mentioned problem is to establish a relation  
\beam
\label{sup.b}
P\left(\sup_{0\le t \le 1}X(t)>u \right) \sim aP\left(X(1)>u \right)\quad\mbox{as $u\to\infty$}
\eeam
where $a$ is a constant.  Then L\'evy-Khinchin formula allows to derive the asymptotics of the right hand 
side probabilities  by powerful analytical tools.

The first result of type (\ref{sup.b}) is L\'evy theorem, which states that for  Brownian
motion  $B(t)\,, t\ge 0$ the following holds:
\beam
\label{B.sup}
P\left(\sup_{0\le t \le 1}B(t)>u \right) = 2P\left(B(1)>u \right)
\eeam
for all $u>0$. During recent years (\ref{sup.b})
was established for various classes of L\'evy processes (see \cite{bib:1}--\cite{bib:5},
\cite{bib:8}, \cite {bib:9},  \cite{bib:10}). 
One of the methods used in these studies is to represent the process in the form
$X(t) = Y(t)+Z(t)$, where $Y(t)$ and $Z(t)$ are independent, $Z(t)$ is a compound Poisson process 
and $Y(t)$ is a L\'evy process
for which $E\exp(c|Y(t)|)<\infty$ for each $c>0$. Assuming the distribution of the jumps of $Z(t)$ 
to be {\it heavy}, (subexponential
or exponential), one first establishes  (\ref{sup.b})  for this process.
Such distributions possess the following property: if $X$ and $Y$ are independent random variables, 
the tail of $X$ is heavy    
and $P(Y>u) = o(P(X>u))$ as $u\to\infty$, then $P(X+Y>u)\sim bP(X>u)$ as $u\to\infty$, where $b$ is a constant.
Using it, on can pass from $Z(t)$ to $X(t)$ (see, for example,
\cite{bib:5} and \cite{bib:9} and references threin).   

But such approach does not work if jumps have a {\it light tail} in the sense of \cite{bib:4}.
So, other methods are called for.

In what follows $C$ denotes a generic constant
which value may vary from line to line. As usually, $F_Y$ stands for the disrtibution of a random variable $Y$.
Througout the paper $\{X_k\}_{k=1}^\infty$ are iid random variables, $S_k=X_1+\cdots+X_k,\,k\ge1,\, S_0=0$. 
\section{Results}
\label{s:R}

We say that the distribution  of a random variable $X$ has  {\it light} right tail if one of the following conditions holds:
\beam
\label{light1}
P(X_1>u)>0\quad \mbox{for all $u>0$\,\, and}\quad\lim_{u\to\infty}\frac{P(X_1>u)}{P(X_1+X_2>u)}=0\,,
\eeam
where $X_1$ and $X_2$ are independent copies of $X$, or 
\beam
X \le A \quad\mbox{a.s. and $P(X>\alpha)>0$}
\label{light2}
\eeam
for  positive constants $A$ and $\alpha$. 

It is known that $X$ has a light tail if and only if $X^+:=\max\{X,0\}$ has it 
(see \cite{bib:4}, Lemma 2).

In what follows we assume that 
\beam
\label{rho.1}
\rho((0,\infty))>0\,
\eeam
and
\beam
\label{rho.2}
\rho((-\infty, 0))<\infty\,.
\eeam
Clearly, (\ref{rho.1}) implies $\rho((a,\infty))>0$ for some  positive $a$.  The third assumption is: 
\beam
\label{rho.3}
\mbox{for some $a>0$ the disribution function}\quad F_{\rho}(x) = 1 - \frac{\rho((\min\{x,a \},\infty))}{\rho((a,\infty))}
\quad\mbox{has light tail.}
\eeam
Our main result is the following
\begin{theorem}
\label{t:main}
Let $\sigma>0$  and (\ref{rho.1})--(\ref{rho.3}) hold. Then for each $b\in\mathbf R$
\beam
\label{maineq}
\lim_{u\to\infty}\frac{P(\sup_{0\le t\le1}X(t)>u)}{P(X(1)>u)}=1\,.
\eeam
\end{theorem}

It should be mentioned that under the additional condition:
$$
\lim_{u\to\infty}\frac{1-F_{\rho}(x+c)}{1-F_{\rho}(x)} = e^{-\alpha c} \quad\mbox{for any real $c$ and  a constant 
$\alpha>0$,} 
$$
 and whithout assumption  (\ref{rho.2}), this statement was proved in
\cite{bib:44} and \cite{bib:1}.

In the case $\sigma=0$ and $\rho(\mathbf R)<\infty$ the process is  compound Poisson with drift. 
It is known that (\ref{maineq}) holds
 for such processes with $b\le 0$, but this limit  
 not always exists if $b>0$ (see \cite{bib:4}). Our next result gives a condition
under which this relation holds for $b>0$ in the absence of Brownian component.

Assume that $P(X>u)>0$ for all $u>0$ and
\beam
\label{cond.pl}
   \lim_{u\to\infty}\frac{P(X>u+a)}{P(X>u)}=0\quad \mbox{for a constant $a>0$}.
\eeam
For independent copies $X_1$ and  $X_2$ of  $X$ we have  
$P(X_1+X_2>u)\ge P(X_1>u-a)P(X_2>a)$, which implies (\ref{light1}). Hence the right tail of $X$ 
is light. 

If the tail of $X$ is given in the form
\beam
\label{tail.X1}
1-F_X = \exp\left(-\int_0^uh(v)dv \right)\,,\, u>u_0\,, 
\eeam
 where $u_0\ge 0$ is a constant and $h$ is a positive  function on $(u_0,\infty)$ 
such that
\beam
\label{cond.inf}
h(v)\to\infty \quad\mbox{ as $v\to\infty$},
\eeam
then (\ref{cond.pl}) holds and, therefore, $X$ has a light tail.

\begin{theorem}
\label{t:pos}
Assume (\ref{rho.1})--(\ref{rho.3}) hold, $\sigma=0$ and the function $F_\rho$ from
(\ref{rho.3}) can be represented in the form (\ref{tail.X1}) with (\ref{cond.inf}). 
Assume also that the function $h$ is continuous,  increasing and satisfies the condition
\beam
\label{cond.h}
h(v + b) \le \exp\left(\frac{bh(v)}8 \right)
\eeam
for $v$ large enough. Then (\ref{maineq}) holds for each $b\in \mathbf R$.
\end{theorem}

Condition (\ref{cond.h}) means that the function $h(v)$ cannot grow too fast as $v\to\infty$.
If $h(v)= \exp(g(v))$ and $g(v+a)\le C(a)g(v)$ for  positive $a$ and $v$, then
(\ref{cond.h}) holds. Another examples are $h(v) = v^c$,
and $h(v) = [\log(v+1)]^c$, where $c$ is a positive constant.
 In can be easily  verified that  $X_1$ with a normal distiribution satisfies the conditions of the theorem.
 Therefore, (\ref{maineq}) holds for compound Poisson
processes with normal jumps and negative drifts. 

As it was shown in \cite{bib:4}, relation (\ref{sup.b}) does not hold  if $X(t)$ 
is a compound Poisson process with negative drift and 
jumps having a lattice 
distribution bounded from above. 
The following result shows that the condition of boundedness can be ommited.

\begin{theorem}
\label{t:neg}
Let (\ref{def.X}) hold with $\sigma=0$,  
and jumps $X_k$ having a lattice distribution with a minimal step $a$. Assume that 
\beam
\label{cond.pl.lat}
P(X_1>na)>0\quad\mbox{for all $n\in {\mathbf N}$\,\, and}\quad    \lim_{n\to\infty}\frac{P(X_1>(n+1)a)}{P(X_1>na)}=0\,.
\eeam 
Then for each $b>0$
\beam
\label{limsup}
\limsup_{u\to\infty}\frac{P\left(\sup_{0\le t\le 1}X(t)>u \right)}{P(X(1)>u)} = \infty
\eeam
and 
\beam
\label{liminf}
\liminf_{u\to\infty}\frac{P\left(\sup_{0\le t\le 1}X(t)>u \right)}{P(X(1)>u)} = 1\,.
\eeam
\end{theorem}

\begin{remark}
\label{r:discr}
{\rm One can obtain a lattice distribution by a "discretization". Namely, for a random variable $X$  
and a fixed $a>0$ put
$$
X^{(a)} = \sum_{n=-\infty}^\infty naI_{(na\le X<(n+1)a)}\,. 
$$
Assume now that the distribution of the  jumps $X_k$ satisfies (\ref{tail.X1})--(\ref{cond.h}).
 Denote by $X_k^{(a)}$ the discretizations of $X_k$, and by $Z^{(a)}(t)$ the corresponding compound Poisson
process given by (\ref{comp.P}). Let $b>0$. Then for the process $X(t)= Z(t)-bt$  we have (\ref{maineq}), while for the
process $X^{(a)}(t) = Z^{(a)}(t)-bt$ relations (\ref{limsup}) and (\ref{liminf}) hold. 
For example, it is true if the jumps
$X_k$ are normal.

The situation is different when the tail of jumps is "heavy", i.e. if
$$
\lim_{u\to\infty}\frac{P(X_1>u+a)}{P(X_1>u)} =1
$$
for any $a>0$. It is known that under this assumption (\ref{maineq}) holds for the process $X(t)$
(see   \cite{bib:10}). Because in this case the tail 
of "discretized" jumps  $X_k^{(a)}$
is also heavy, (\ref{maineq}) holds  for the process $X^{(a)}(t)=Z^{(a)}(t)-bt$ also.}
\end{remark}

Theorems \ref{t:main} and \ref{t:neg} show that sometimes the process $Z(t)-bt$ does not satisfy (\ref{sup.b}),
while for the process (\ref{def.X}) with $\sigma>0$ relation (\ref{maineq}) holds.  
Our last result states that a compound Poisson process $Z(t)$ may  satisfy (\ref{maineq}),
while for the process $X(t)= \sigma B(t)+Z(t)$ relation  (\ref{sup.b}) does not hold. Clearly, if the jumps
of $Z(t)$ are positive, then its  supremum over $[0,1]$ is $Z(1)$.

\begin{theorem}
\label{t:example}
There is  a compound Poisson process $Z(t)$ with positive jumps such that for the process (\ref{def.X}) with 
$\sigma>0$ and $b=0$
\beam
\label{limsup}
\limsup_{u\to\infty}\frac{P(\sup_{0\le t\le1}X(t)>u)}{P(X(1)>u)}>1
\eeam
and
\beam
\label{liminf}
\liminf_{u\to\infty}\frac{P(\sup_{0\le t\le1}X(t)>u)}{P(X(1)>u)}=1\,.
\eeam
\end{theorem}

\section{Auxiliary statements}
\label{S:lem}

Here we prove some statements that are used later . 
\begin{lemma}
\label{l:comp}
Let $Z$ and $W$ be random variables,  $P(Z>u)>0$ and $P(W>u)>0$ for all positive $u$, and one of the following
conditions holds:
\beam
\label{eq}
\lim_{u\to\infty}\frac{P(Z>u)}{P(W>u)}=1\,,
\eeam
or 
\beam
\label{neg}
\lim_{u\to\infty}\frac{P(Z>u)}{P(W>u)}=0\,.
\eeam
Let a random variable $Y$ satisfy  condition (\ref{cond.pl}).
If $Y$ is independent of $Z$ and $W$, then
$$
\lim_{u\to\infty}\frac{P(Y+Z>u)}{P(Y+W>u)}=1\,
$$
if (\ref{eq}) holds, and
$$
\lim_{u\to\infty}\frac{P(Y+Z>u)}{P(Y+W>u)}=0\,
$$
if (\ref{neg}) holds. 
\end{lemma}
\begin{proof}
If (\ref{eq}) holds, then
for a fixed $\ep>0$ we can find  $u_0>0$ such that 
$$
P(Z>u)\le (1+\ep)P(W>u)
$$
 for all $u\ge u_0$. Hence
$$
P(Y+Z>u) \le (1+\ep)\int_{-\infty}^{u-u_0}P(W>u-t)F_Y(dt) + P(Y>u-u_0)
$$
$$
\le (1+\ep)P(Y+W>u) + P(Y>u-u_0)\,.
$$
We also have $P(Y+W>u) \ge P(Y>u-u_0-a)P(W>u_0+a)$.
From here and (\ref{cond.pl})
$$
\limsup_{u\to\infty}\frac{P(Y+Z>u)}{P(Y+W>u)}\le(1+\ep)\,.
$$
But by the same way
$$
\limsup_{u\to\infty}\frac{P(Y+W>u)}{P(Y+Z>u)}\le(1+\ep)\,.
$$
Letting $\ep\to 0$ we get the first needed relation. The second one can be obtained similarly. 
\end{proof}

It is known that for compound Poisson process  with light tail relation (\ref{sup.b}) holds with $a=1$
(see Theorem 1 from \cite{bib:4}). Because the random variable $Y=B(1)$ satisfies (\ref{cond.pl}),
 we come to the following statement.

\begin{corollary}
\label{c:equiv}
If the jumps of a compound Poisson process $Z(t)$ have a light tail, and $B(1)$ is independent
of this process, then
\beam
\label{equiv}
\lim_{u\to\infty}\frac{P(B(1)+\sup_{0\le t\le1}Z(t)>u)}{P(B(1)+Z(1)>u)}=1\,.
\eeam
\end{corollary}

\begin{lemma}
\label{l:Y}
Let random variables $\{X_k\}_{k=1}^\infty$ and $Y$ be independent. Assume that  the tail of $X_k$ is light  and 
$Y$ satisfies (\ref{cond.pl}). Assume also that $a\le\alpha$ in the case (\ref{light2}).  Then
$$
\lim_{u\to\infty}\frac{P(Y+S_k>u)}{P(Y+S_{k+1}>u)}= 0
$$
for all $k=1,2,\dots$.
\end{lemma}
\begin{proof}
Fist we consider  case (\ref{light1}). Then, according to Lemma 4 from \cite{bib:4} 
\beam
\label{Sk.0}
\lim_{u\to\infty}\frac{P(S_k>u)}{P(S_{k+1}>u)}= 0\,
\eeam
and Lemma \ref{l:comp} leads to the needed conclusion.

Turn to  case (\ref{light2}). Then $S_k\le Ak$ and 
$$
P(Y+S_k>u) =\int_{-\infty}^{Ak}P(Y>u-t)F_{S_k}(dt)\,,
$$
$$
P(Y+S_{k+1}>u) =\int_{-\infty}^{Ak}P(Y+X_1>u-t)F_{S_k}(dt).
$$
We have $P(Y+X_1>u-t)\ge P(Y>u-t-a)P(X_1>a)$, and $P(X_1>a)>0$ because $a\le\alpha$.
It follows from this estimate and (\ref{cond.pl}) that for a fixed $\ep>0$ there is $u_0>0$ such that if $u-t>u_0$, 
then
$$
\frac{P(Y>u-t)}{P(Y+X_1>u-t)}\le \ep\,.
$$
But in the last integrals $t\le Ak<u-u_0$, i.e. $u-t>u_0$ for $u$ large enough. For such $u$ 
$$
\frac{P(Y+S_k>u)}{P(Y+S_{k+1}>u)} \le \ep\,.
$$
Letting $u\to\infty$ and then $\ep\to0$ we obtain the lemma.
\end{proof}

The next  statement plays an important role in the proof of Theorem \ref{t:main}.

\begin{lemma}
\label{l:sym}
Assume random variables $X$ and $Y$ are independent and $Y$ is symmetric. Then 
$$
P(X+|Y|>u) = 2P(X+Y>u) - P(X>u+|Y|)\,.
$$
for all $u>0$\,.
\end{lemma}
\begin{proof}  
We have
$$
P(X+|Y|>u) = P(X+Y>u) + P(X+Y\le u\,,\, X+|Y|>u) = P(X+Y>u)
$$
$$
+ P(X>u\,,\, X+Y\le u) +  P(X\le u\,,\,X+Y\le u\,,\, X+|Y|>u)\,.
$$
Because of symmetry and independence
$$
P(X>u\,,\, X+Y\le u) = P(X>u\,,\, Y\le u-X) = P(X>u\,,\, Y\ge X-u) \,.
$$
By the same reasons
$$
P(X\le u\,,\,X+Y\le u\,,\, X+|Y|>u) = P(X\le u\,,\,Y\le u-X\,,\, |Y|>u-X)
$$
$$
= P(X\le u\,,\,Y\le u-X\,,\, -Y>u-X)= P(X\le u\,,\, Y>u-X) 
$$
$$
= P(X\le u\,,\,X+Y>u) = P(X+Y>u) - P(X> u\,,\,X+Y>u)\,.
$$
Inserting the last two relation in the first one we get
$$
P(X+|Y|>u) = 2P(X+Y>u) + P(X>u\,,\, Y\ge X-u) -P(X> u\,,\,Y>u-X)
$$
$$
= 2P(X+Y>u) -P(X>u\,,\,u-X <Y< X-u)\,.
$$
Since the last probability is
$$
P(X>u,\, |Y|< X-u) = P(X>u+|Y|)\,,
$$
the lemma follows.
\end{proof}

The following lemma will allow us to reduce the proofs of Theorems \ref{t:main} and \ref{t:pos}
to the case of processes of type  (\ref{def.X}).

\begin{lemma}
\label{l:gen}
Let 
\beam
\label{rep.sub}
X(t) = X_1(t) +X_2(t)\,,
\eeam
where L\'evy processes $X_1(t)$ and $X_2(t)$ are independent, $X_2(t)$ is a subordinator
with  L\'evy measure $\rho_2$ such that $\rho_2((a_2,\infty))=0$, where $a_2>0$ is a constant.
Assume that $\rho_1((a_1,\infty))>0$ for $a_1>a_2$, where  $\rho_1$ is the L\'evy measure of $X_1(t)$.
Assume also that $X_1(t)$ satisfies (\ref{maineq}). Then this relation holds for the process $X(t)$.  
\end{lemma}

\begin{proof}
Since $X_2(t)$ is a subordinator, then
\beam
\label{first}
P\left(\sup_{0\le t\le 1}X(t)>u\right) \le P\left(\sup_{0\le t\le 1}X_1(t)+X_2(1) >u\right)
\eeam
$$
\le \int_{-\infty}^{u-A}P\left(\sup_{0\le t\le 1}X(t)>u-v\right)F_{X_2(1)}(dv) + 
P(X_2(1)>u-A)\,,
$$
where $A$ is a positive constant. Because $X_1(t)$ satisfies (\ref{maineq}), for a fixed
$\ep >0$ there is $A$ such that the integral does not exceed 
$$
(1+\ep)P(X_1(1)+X_2(1) >u) = (1+\ep)P(X(1) >u)\,. 
$$
It is well known  that the conditions $\rho_1((a_1,\infty))>0$ and $\rho_2((a_2,\infty))=0$ for $a_1>a_2$ 
implies  
$$
P(X_2(1)>u-A) = o\left(P(X_1(1)>u)\right) 
$$
for any positive $A$ as $u\to\infty$ (see \cite{bib:77}). From here and (\ref{first})
$$
\limsup_{u\to\infty}\frac{P(\sup_{0\le t\le1}X(t)>u)}{P(X(1)>u)} \le (1+\ep)
$$
for each $\ep>0$, and the lemma follows. \end{proof}


Let $b>0$ and $Z(t)$ be  defined by (\ref{comp.P}). 
Denote by $\Gamma_k\,,k\ge 1$, the arrival times of $Z(t)$ and put $\Gamma_0=0$. Let
\beam
\label{def.tau}
\tau = \max\{k: \Gamma_k<1\}\,.
\eeam  
Let 
\beam
\label{def.m}
m = \min\{k: P(S_k>b)>0\}\,,
\eeam
and
\beam
\label{def.ak}
a_k = \max\left\{1 - \frac{(m+1)\log k}k\,,\,0 \right\}.
\eeam
Put
\beam
\label{def.Q}
Q(u) = P\left(Z(1)>u+b\Gamma_{\tau}\,, \Gamma_{\tau}\le a_{\tau}\right)
\eeam

\begin{lemma}
\label{l:int}
If $X_k$ have a light tail, then for any $b>0$
\beam
\label{lim.Q}
\lim_{u\to\infty}\frac{Q(u)}{P(Z(1)>u+b)} = 0\,.
\eeam
\end{lemma}

\begin{proof}
It can be easily verified that
$$
Q(u) = \la e^{-\la}\sum_{k=1}^\infty\int_0^{a_k}\frac{(\la t)^{k-1}}{(k-1)!}P(S_k>u+bt)dt\,.
$$
Fix an index $M$ and denote
$$
Q_M(u) = \la e^{-\la}\sum_{k=1}^M\int_0^{a_k}\frac{(\la t)^{k-1}}{(k-1)!}P(S_k>u+bt)dt\,,\quad Q^{(M)}(u)
= Q(u) - Q_M(u)\,.
$$
It is clear that
$$
P(Z(1)>u+b) = e^{-\la}\sum_{k=1}^\infty\frac{\la^k}{k!}P(S_k>u+b)\,
$$
and, therefore, for each $k$
\beam
\label{tail.cP}
P(Z(1)>u+b) > \frac{\la^{k+m}}{(k+m)!}P(S_{k+m}>u+b) > \frac{\la^{k+m}}{(k+m)!}P(S_m>b)P(S_{k}>u)\,.
\eeam
Since
$$
Q_M(u) \le e^{-\la}\sum_{k=1}^M\frac{\la^k}{k!}P(S_k>u),
$$
the last estimate and (\ref{Sk.0}) imply
$$
\lim_{u\to\infty}\frac{Q_M(u)}{P(Z(1)>u+b)} = 0\,.
$$
Further, denoting
$$
\delta(k,u)= \frac{\la \int_0^{a_k}\frac{(\la t)^{k-1}}{(k-1)!}P(S_k>u+bt)dt}{\frac{\la^{k+m}}{(k+m)!}P(S_{k+m}>u+b)}\,,
$$ 
we see that
$$
\frac{Q^{(M)}(u)}{P(X(1)>u)} \le \sup_{k>M}\delta(k,u)\,, 
$$
and
$$
\delta(k,u) \le \frac{\frac{(\la a_k)^k}{k!}P(S_k>u)}{\frac{\la^{k+m}}{(k+m)!}P(S_k>u)P(S_m>b)}
= \frac{(k+1)\cdots(k+m)a_k^k}{\la^mP(S_m>b)}\,.
$$
Hence
$$
\limsup_{u\to\infty}\frac{Q(u)}{P(X(1)>u)} = \limsup_{u\to\infty}\frac{Q^{(M)}(u)}{P(X(1)>u)} 
\le \sup_{k>M} \frac{(k+1)\cdots(k+m)a_k^k}{\la^mP(S_m>b)}\,.
$$
According to (\ref{def.ak}) $(k+1)\cdots(k+m)a_k^k\to0$ as $k\to\infty$. Hence, letting $M\to\infty$
we come to (\ref{lim.Q}).
\end{proof}

We also will use the following well known estimate for the normal distribution. If $Y$ is normal with
mean zero and variance one, then for all $x>1$
\beam
\label{norm}
\frac{1}{\sqrt{2\pi}}\left(\frac{1}{x} -\frac{1}{x^3} \right)\exp\left(-\frac{x^2}{2}\right) \le P(Y>x)
\le \frac{1}{\sqrt{2\pi}}\frac{1}{x}\exp\left(-\frac{x^2}{2}\right)\,.
\eeam

\section{Proof of Theorem \ref{t:main}}
\label{s:pr.1}

According to (\ref{rho.2}) we can represent our process in the form (\ref{rep.sub}), 
 and Lemma \ref{l:gen} shows that it is enough to prove the theorem for the process of type (\ref{def.X}).

If $b<0$, then 
$$
P\left(\sup_{0\le t\le 1}X(t)>u\right) \le P\left(\sup_{0\le t\le1}[B(t)+Z(t)]  >u+b\right)\,.
$$
So, (\ref{maineq}) for $b=0$ implies the same relation for $b<0$. Hence, we may assume $b\ge0$ in the sequel. 
Without loss of generality $\sigma=1$.

Let $\tau$ be given by (\ref{def.tau}). Then 
$$
P\left(\sup_{0\le t \le 1}X(t)>u \right) \le P\left(\sup_{0\le t < \Gamma_{\tau}}X(t)>u \right)
+ P\left(\sup_{\Gamma_{\tau}\le t \le 1}X(t)>u \right)
$$
$$
= A(u) +C(u)\,.
$$
The theorem will follow from the next two equalities:
\beam
\label{est.C}
\lim_{u\to\infty}\frac{C(u)}{P(X(1)>u)} =1
\eeam
and
\beam
\label{est.A}
\lim_{u\to\infty}\frac{A(u)}{P(X(1)>u)} =0\,.
\eeam

\subsection{Proof of (\ref{est.C}).} Let $\widetilde B(t)$ be a Brownian motion independent of $X(t)$. We have 
\label{ss:C}
$$
C(u) = P\left(\sup_{\Gamma_{\tau}\le t\le 1} \left[B(\Gamma_{\tau}) +S_{\tau} -bt + \widetilde B(1-t) \right]>u \right)
$$
$$
\le P\left(B(\Gamma_{\tau}) +S_{\tau} -b\gt + |\widetilde B(1-\Gamma_{\tau})| >u \right)
$$
because  L\'evy formula (\ref{B.sup}) can be written in the form
\beam
\label{sup.B} 
\sup_{0\le t\le 1}B(t) \eid |B(1)|\,. 
\eeam
 Applying Lemma \ref{l:sym}
conditionally on $\Gamma_{\tau}$ and taking into account the relations 
$$
B(\Gamma_{\tau})+\widetilde B(1-\gt) \eid B(1)\quad\mbox{and}\quad S_\tau= Z(\gt) = Z(1)\,, 
$$ 
we conclude that
\beam
\label{C(u)}
C(u)\le 2P\left(B(1) + Z(1)  >u + b\gt \right)
- P\left(B(\Gamma_{\tau}) +Z(1) >u + b\gt +|\widetilde B(1-\gt)| \right).
\eeam
To obtain (\ref{est.C}) it is enough to prove the following 
\begin{lemma} 
\label{l:B}
For each $b\ge 0$:
\beam
\label{liminf.P}
\liminf_{u\to\infty}
\frac{P\left(B(\gt)+Z(1)>u+b\Gamma_{\tau}+ |\widetilde B(1-\Gamma_{\tau})|\right)}{P(B(1)+Z(1)>u+b)} \ge 1\,,
\eeam
and
\beam
\label{lim.P}
\lim_{u\to\infty}\frac{P(B(1)+Z(1)>u+b\Gamma_{\tau})}{P(B(1)+Z(1)>u+b)} = 1\,.
\eeam
\end{lemma}

\noindent{\bf Proof of (\ref{liminf.P}).}
Because 
$$
P\left(B(\gt)+Z(1)>u+b\Gamma_{\tau}+ |\widetilde B(1-\Gamma_{\tau})|\right) 
\ge P\left(B(\gt)+Z(1)>u+b+ |\widetilde B(1-\Gamma_{\tau})|\right)\,,
$$
it is enough to establish (\ref{liminf.P}) for $b=0$. 

\noindent{\it Step 1.} Fix a $\delta\in (0,1)$. There is a positive constant $D$ such that
$$
P(|B(1)|<D)>1-\delta\,.
$$ 
Since 
$$
\widetilde B(1-t)\eid \sqrt{1-t}B(1)\,,\, 0<t<1\,, 
$$
we get for each $t\in (0,1)$ and $k\in \mathbf N$ 
$$
P\left( B(t)+S_k>u + |\widetilde B(1-t)|\right) \ge P\left( B(t)+S_k>u +\drt\right)P\left(|\widetilde B(1-t)| < \drt \right)
$$
$$
\ge (1-\delta)P\left(B(t)+S_k>u+\drt \right)\,.
$$
Integrating with respect to $\Gamma$-densities and summing up over $k$'s we obtain
\beam
\label{def.H}
P(X(\Gamma_\tau ) >u+|\widetilde B(1-\Gamma_\tau )|) 
\eeam
$$
\ge (1-\delta) e^{-\la}\sum_{k=1}^\infty\frac{\la^k}{(k-1)!}
\int_0^1 P\left(B(t)+S_k>u+\drt \right)t^{k-1}dt 
$$
$$
=(1-\delta )P(X(\Gamma_\tau ) >u+D\sqrt{1-\Gamma_{\tau}} ) :=(1-\delta)H(u)\,.
$$
To prove (\ref{liminf.P}) for $b=0$ it is enough to establish that for each $D>0$
 \beam
 \label{inf.H}
\liminf_{u\to\infty}\frac{H(u)}{P(X(1)>u)} \ge 1\,.
\eeam

\noindent{\it Step 2.} From now on $\alpha$ is a positive  constant for which
\beam
\label{def.alpha}
P(X_1>\alpha) >0\,,
\eeam
and $a$ is a constant such that
\beam
\label{def.a}
 a > \max\left\{1, \frac1{\alpha} \right\}\,.
 \eeam

For fixed $T\in \bf{N}$ and $u>0$, where $2\le T < au$, we divide $\bf N$ into three parts:
\beam\label{Ni}
{\bf N}_1(T,u) = \{k: k\le T\}\,,\, {\bf N}_2(T,u) = \{k:  T< k \le [au]\},\,   
\eeam
$$
 {\bf N}_3(T,u) = \{k:   k>[au]\}\,. 
$$ 
Using (\ref{tail.X}) and denoting by $G_i(u)\,,i=1,2,3$, the sums of summands over ${\bf N}_i(T,u)$ we may write 
\beam
\label{tail.rep} 
P(X(1)>u) = e^{-\lambda}P(B(1)>u) + G_1(u) + G_2(u)+G_3(u)\,.
\eeam
It follows from Lemma \ref{l:Y} that for each $T\in {\bf N}$
\beam
\label{G1}
\lim_{u\to\infty}\frac{e^{-\lambda}P(B(1)>u) + G_1(u)}{P(X(1)>u)} =0\,.
\eeam

Now we show that
\beam
\label{G3}
\lim_{u\to\infty}\frac{G_3(u)}{P(X(1)>u)} =0\,.
\eeam
Indeed, according to Stirling formula
$$
G_3(u) \le \frac{\la^{[au]+1}}{([au]+1)!} = \exp\left(-a(u\log u)(1+ g(u)) \right)\,,
$$
where $g(u)\to1$ as $u\to\infty$. 

On the other hand, for 
$$
k(u) = \max\left\{[u]\,,\left[\frac{u}{\alpha}\right] \right\}
$$
we have, once again applying Stirling formula,
$$
P(X(1)>u)\ge e^{-\la}\frac{\la^{k(u)}}{k(u)!}P(B(1)+S_{k(u)}>u) 
$$
$$
\ge e^{-\la}\frac{\la^{k(u)}}{k(u)!}P(B(1)>0)P(X_j>\alpha\,,1\le j \le k(u))
$$
$$
= \exp\left(-k(u)\log k(u)(1+g_1(u) \right)\,,
$$
where, as above, $g_1(u)\to1$ as $u\to\infty$. According to (\ref{def.a})
$$
k(u) \le u\max\left\{1\,,\frac{1}{\alpha} \right\} <au\,,
$$
and (\ref{G3}) follows from here and the last two estimates.

So, for each $T\in {\bf N}$
\beam
\label{G2}
\lim_{u\to\infty}\frac{G_2(u)}{P(X(1)>u)} =1\,.
\eeam

\noindent{\it Step 3.} Here we represent $G_2(u)$ as a sum of two quantities, such that
the first of them is small relative to $P(X(1)>u)$. Denote
\beam
\label{def.ga} 
g_a(k,u) = u - a\log(\min\{k,u\})
\eeam
and
\beam
\label{def.I}
I(k,u) = P\left(B(1)+S_k>u\,,S_k\le g_a(k,u)  \right)\,.
\eeam
Put
\beam
\label{G21}
G_{21}(u) = e^{-\la}\sum_{k=T+1}^{[au]}\frac{\la^k}{k!}I(k,u)\,,\quad G_{22}(u)= G_{2}(u) -G_{21}(u)\,.   
\eeam

\begin{proposition}
\label{p:21}
For $u>T$ the following inequality holds:
\beam
\label{est.21}
\frac{G_{21}(u)}{P(X(1)>u)} \le \frac{2aCe^{\alpha^2/2}}{\la P(X_1>\alpha)}T^{1-\alpha a}\,, 
\eeam
where $C$ is an absolute constant.
\end{proposition}
\begin{proof}
We have, using (\ref{tail.X}),
\beam
\label{G21.2}
\frac{G_{21}(u)}{P(X(1)>u)} \le \max_{T+1\le k\le au}
\frac{\frac{\la^k}{k!}I(k,u)}{\frac{\la^{k+1}}{(k+1)!}P(B(1)+S_{k+1}>u)} 
\eeam 
$$
= \max_{T+1\le k\le au}\frac{(k+1)I(k,u)}{\la P(B(1)+S_{k+1}>u)}\,.
$$
Further,
$$
P(B(1)+S_{k+1}>u)\ge P(B(1)+S_{k}>u-\alpha)P(X_1>\alpha)\,,
$$
which yields
$$
\frac{I(k,u)}{P(B(1)+S_{k+1}>u)}\le \frac{1}{P(X_1>\alpha)}\frac{I(k,u)}{P(B(1)+S_{k}>u-\alpha)}
$$
$$
\le \frac{1}{P(X_1>\alpha)}
\frac{\int_{-\infty}^{g_a(k,u)}P(B(1)>u-y)F_{S_k}(dy)}{\int_{-\infty}^{g_a(k,u)}P(B(1)>u-y-\alpha)F_{S_k}(dy)}
$$
$$
\le \frac{1}{P(X_1>\alpha)}\max_{y\le g_a(k,u)}\frac{P(B(1)>u-y)}{P(B(1)>u-y-\alpha)}\,.
$$
If $y\le g_a(k,u)$, then $u-y \ge u-g_a(k,u)= a\min\{\log k, \log u\}$. Since $k,u >T$, we obtain
using (\ref{norm}) and elementary computations,
$$
\frac{P(B(1)>u-y)}{P(B(1)>u-y-\alpha)} \le C\exp\left( -\alpha(u-y) +\frac{\alpha^2}{2}\right)
\le Ce^{\alpha^2/2}\exp\left(-\alpha a\min\{\log k, \log u \} \right)\,,
$$
where $C$ is a constant independent of $k$ and $u$. Because $\alpha a >1$, this inequality jointly with
previous ones give us (\ref{est.21}).
\end{proof}

\noindent{\it Step 4.} Define
\beam
\label{def.J}
J(k,u) =e^{-\la}\frac{\la^k}{(k-1)!}\int_0^1P\left(B(t)+S_k>u+\drt\,,\,S_k>g_a(k,u) \right)t^{k-1}dt\,. 
\eeam
The following statement is the main part of our proof.
\begin{proposition}
\label{p:J}
For each $\ep\in(0,1)$ and $D>0$ there are $T_0\in {\bf N}$ and $u_0>0$ such that 
\beam
\label{gamma.k}
\gamma_k(u) := \frac{J(k,u)e^{\la}k!}{P(B(1)+S_k>u\,,\,S_k>g_a(k,u))\la^k} >1-\ep
\eeam
for all $k>T_0$ and $u>u_0$.
\end{proposition}

\begin{proof}
We can write
$$
J(k,u)=e^{-\la}\frac{\la^k}{(k-1)!}\int_{g_a(k,u)}^\infty\int_0^1 
P\left(B(t)>u-y +\drt \right)t^{k-1}dtF_{S_k}(dy)
$$
and
$$
P(B(1)+S_k>u\,,\,S_k>g_a(k,u)) =\int_{g_a(k,u)}^\infty P\left(B(t)>u-y  \right)F_{S_k}(dy)\,,
$$
which yields that
\beam
\label{g.l}
\gamma_k(u) \ge k\min_{y> g_a(k,u)}\frac{\int_0^1P\left(B(t)>u-y +\drt \right)t^{k-1}dt}{P(B(1)>u-y)}\,.
\eeam

We estimate the expression in the right hand side dividing the area $[g_a(k,u), \infty)$  into three parts:
$[g_a(k,u), u-\beta),\,[u-\beta, u+\beta_1)$ and $[u+\beta_1,\infty)$, where positive constants
$\beta$ and $\beta_1$ will be choosen later. We also denote by $\gamma_k^{(1)}(u),\,
\gamma_k^{(2)}(u)$ and $\gamma_k^{(3)}(u)$ the minima over these parts correspondingly.

\noindent{\it Case 1: $g_a(k,u)\le y < u-\beta$.} We assume $\beta >1$. Estimate (\ref{norm}) implies that
$$
\nu(t, u-y) := \frac{P\left(B(t)>u-y +\drt \right)}{P(B(1)>u-y)}
$$
$$
\ge \frac{\beta}{\beta +D}\left(1 - \frac1{\beta^2}\right)\exp\left(\frac{(u-y)^2}2 - \frac{(u-y+\drt)^2}{2t} \right)
$$
$$
= \frac{\beta^2-1}{\beta(\beta+D)}\exp\left(-\frac{(u-y)^2(1-t)}{2t} - \frac{(u-y)\drt}{t} -
\frac{D^2(1-t)}{2t} \right)\,.
$$
Fix $b>0$. If $1 -\frac{b}{k}<t<1$, then
$$
\nu(t, u-y) \ge \frac{\beta^2-1}{\beta(\beta+D)}\exp\left(-\frac{(u-y)^2b}{2(k-b)} - \frac{D(u-y)\sqrt{bk}}{k-b} -
\frac{D^2b}{2(k-b)} \right)\,.
$$
Denoting
\beam
\label{def.xi}
\xi_b(x) = \exp\left(-\frac{a^2b(\log x)^2}{2(x-b)} - \frac{a\sqrt bD\sqrt x\log x}{x-b} -
\frac{bD^2}{2(x-b)} \right)\,
\eeam
and taking into account that $\beta < u-y < a\log(\min\{k,u\})$, we obtain
$$
\nu(t, u-y) \ge \frac{\beta^2-1}{\beta(\beta+D)}\xi_b\left( \min\{k,u\}\right)\,.
$$
Restriction of the area of integration in (\ref{g.l}) to $1 -\frac{b}{k}\le t <1$ yields
\beam
\label{gk.1}
\gamma_k^{(1)}(u)\ge \left[1 - \left(1-\frac{b}{k} \right)^k \right]
\frac{\beta^2-1}{\beta(\beta+D)}\xi_b\left( \min\{k,u\}\right)\,.
\eeam

\noindent{\it Case 2: $u-\beta \le y < u+\beta_1$}. Now for $1 -\frac{b}{k}\le t <1$ 
\beam
\label{def.chi}
 \left|\frac{u-y +\drt}{\sqrt{t}} -(u-y)\right| \le \max\{\beta\,,\,\beta_1 \}\frac{b}{2k\left(1-\frac{b}{k}\right)^{3/2}}
+D\sqrt{\frac{b}{k}} :=\chi_b(k) 
\eeam
Hence 
$$
\nu(t, u-y) \ge \frac{P\left(B(1)>u-y +\chi_b(k) \right)}{P(B(1)>u-y)}
= 1- \frac{\int_{u-y}^{u-y+\chi_b(k)}e^{-t^2/2}dt}{\int_{u-y}^{\infty}e^{-t^2/2}dt}
$$
$$
\ge 1 - \frac{\chi_b(k)}{\int_{\beta}^{\infty}e^{-t^2/2}dt} = 1 - \frac{\chi_b(k)}{\sqrt{2\pi}P(B(1)>\beta)}\,.
$$
because $u-y<\beta$. From here, as above
\beam
\label{gk.2}
\gamma_k^{(2)}(u)\ge \left[1 - \left(1-\frac{b}{k} \right)^k \right]
\left(1 - \frac{\chi_b(k)}{\sqrt{2\pi}P(B(1)>\beta)} \right)\,.
\eeam

\noindent{\it Case 3:  $y \ge u+\beta_1$}. Now
$$
\gamma_k^{(3)}(u) \ge k\int_0^1P(B(t)>-\beta_1+D)t^{k-1}dt\,.
$$
Choose $\beta_1>D$. Then
$$
P(B(t)>-\beta_1+D)= P\left(B(1) > \frac{-\beta_1+D}{\sqrt t}\right)
> P(B(1)>-\beta_1+D)
$$
for all $0<t<1$. Hence
\beam
\label{gk.3}
\gamma_k^{(3)}(u) \ge P(B(1)>-\beta_1+D)\,.
\eeam

Now we are able to finish the proof of the proposition. 
Fix $\delta\in (0,1)$ and choose  $b>0$ such that $e^{-b}< \delta$, and
$k_0\in {\bf N}$ for which
$$
1-\left(1-\frac{b}{k}\right)^k > (1-\delta)^2 \quad\mbox{for $k>k_0$}\,.
$$
Choose now $\beta$ under the condition
$$
\frac{\beta^2}{\beta(\beta+D)} >1-\delta \,.
$$  
According to (\ref{def.xi}), $\xi_b(x)\to 1$ as $x\to\infty$ for each $b>0$. Hence, for choosen $b$ there is $u_0>0$ 
such that $\xi_b(\min\{k,u\})>1-\delta$ if $\min\{k,u\}>u_0$. So, (\ref{gk.1}) implies 
$$
\gamma_k^{(1)}(u) > (1-\delta)^4\quad \mbox{for $k>\max\{k_0, u_0\}$ and $u>u_0$}\,.
$$
Further, (\ref{gk.3}) allows us to find $\beta_1$ such that
$\gamma_k^{(3)}(u)> (1-\delta)$.
According to (\ref{gk.2}) and (\ref{def.chi}) for choosen $b$, $\beta$ and $\beta_1$ there exists  
$k_1>k_0$ such that $\gamma_k^{(2)}(u) > (1-\delta)^3$ for all $k>k_1$. Finally,
$$
\gamma_k(u) \ge \min\left\{\gamma_k^{(1)}(u)\,,\,\gamma_k^{(2)}(u)\,,\,\gamma_k^{(3)}(u) \right\}
>(1-\delta)^4
$$
for $k>T_0=\max\{k_1,u_0\}$ and $u>u_0$, and the needed statement follows. 
\end{proof}
\begin{corollary}
\label{c:H2}
Let
\beam
\label{H.22}
H_{22}(u) = \sum_{k=T+1}^{[au]}J(k,u)\,.
\eeam
For each $\ep>0$ there are $T_0\in \bf N$ and $u_0>T_0/a$ such that
$$
\frac{H_{22}(u)}{G_{22}(u)} >1-\ep
$$
for all $u>u_0$ and $T_0\le T < [au]$.
\end{corollary}

\noindent{\it Step 5.} Now we can proof (\ref{inf.H}). Indeed,
(\ref{tail.rep}) allows us to write
$$
\frac{H(u)}{P(X(1)>u)} \ge \frac{H_{22}(u)}{P(X(1)>u)}
= \frac{H_{22}(u)}{G_{22}(u)}\times\frac{1}{\frac{G_{21}(u)}{G_{22}(u)}+1}
\times\frac{1}{\frac{e^{-\lambda}P(B(1)>u)+G_1(u)}{G_2(u)}+1 +\frac{G_3(u)}{G_2(u)}}\,.
$$
Fix $\ep\in(0,1)$. Applying Corollary \ref{c:H2} we see that the first fraction in the right hand side
is greater then $1-\ep$ for $u>u_0$ and $T_0\le T <[au]$, where $u_0$ and $T_0$ are constants. 
 It follows from (\ref{est.21}), (\ref{G21}) and (\ref{G2})
that there is $T_1\in \bf N$ such that $G_{21}(u)/G_{22}(u)<\ep$ for $u>T_1$. Choose now $T>\max\{T_0, T_1\}$.
For such $T$, according to (\ref{G1}) and (\ref{G3}), there is $u_1>0,\, au_1 > T$, such that 
$[e^{-\lambda}P(B(1)>u)+G_{1}(u)]/G_{2}(u)<\ep$
and $G_{3}(u)/G_{2}(u)<\ep$ for $u>u_1$. So, 
$$
\frac{H(u)}{P(X(1)>u)} > \frac{1-\ep}{(1+\ep)(1+2\ep)}
$$
for $u>\max\{T,\,u_0,\,u_1\}$. Letting $u\to\infty$ and then $\ep\to0$ we get (\ref{inf.H}).   \qed

\medskip

\noindent{\bf Proof of (\ref{lim.P}).} Put
\beam
\label{Q.til}
\widetilde{Q}(u) = P\left(B(1)+Z(1)>u+b\Gamma_{\tau}\,, \Gamma_{\tau}\le a_{\tau}\right),
\eeam
$$
R(u) = P\left(B(1)+Z(1)>u+b\Gamma_{\tau})\right) - \widetilde{Q}(u), 
$$
where the numbers $a_k$ are given by (\ref{def.ak}).

Let $Q(u)$ be given by (\ref{def.Q}).
Denote $C= \lim_{u\to-\infty}{Q}(u)$ and $F(u)= 1 -C^{-1}{Q}(u)$. 
Let $W$ be a random variable with distribution
function $F$, independent of $B(1)$. It can be easily verified, using formulas for $\Gamma$-densities, that
$P(B(1)+W>u) = C^{-1}\widetilde{Q}(u)$. Now  (\ref{lim.Q}) and Lemma \ref{l:comp} imply that
\beam
\label{est.Q.til}
\widetilde{Q}(u) = o\left( P(B(1)+Z(1)>u+b)\right)
\eeam
as $u\to\infty$. 

Next we show that
\beam
\label{est.R}
\limsup_{u\to\infty}\frac{R(u)}{P(B(1)+Z(1)>u+b)} \le 1.
\eeam
Fix a constant $a>0$  and denote
\beam
\label{def.R12}
R_1(u) = P\left(B(1)+Z(1)>u+b\Gamma_{\tau}\,,\, \Gamma_{\tau}> a_\tau\,,\, S_{\tau}
> u -a\log\tau \right),
\eeam
$$
R_2(u) = R(u)-R_1(u).
$$

\noindent{\it Estimate for $R_1(u)$.} We show here that
\beam
\label{est.R1}
\limsup_{u\to\infty}\frac{R_1(u)}{P(B(1)+Z(1)>u+b)} \le 1.
\eeam
As above, one can easily check the formula
\beam
\label{rep.R1}
R_1(u) = \la e^{-\la}\sum_{k=1}^\infty\int_{a_k}^1\frac{(\la t)^{k-1}}{(k-1)!}P(B(1)+S_k>u+bt\,,\,
S_k > u -a\log k )dt\,,
\eeam
and it is clear that
\beam
\label{tail.X}
P(B(1)+Z(1)) = e^{-\la}P(B(1)>u) + e^{-\la}\sum_{k=1}^\infty\frac{\la^k}{k!}P(B(1)+S_k>u)\,.
\eeam

Denote
$$
\gamma_k(u) = \frac{k\int_{a_k}^1P(B(1)+S_k>u+bt\,,\,
S_k > u -a\log k )t^{k-1}dt}{P(B(1)+S_k>u+b)}\,.
$$
Representing the probabilities in this fraction as integrals with respect to $F_{S_k}$, we obtain
the estimate
\beam
\label{gamma.u1} 
\gamma_k(u) \le \sup_{y>u-a\log k}\frac{P(B(1)> u+a_kb-y)}{P(B(1)> u+b-y)}\,.
\eeam

\noindent{\it Step 1.} Fix $\ep>0$ and choose $A>0$ such that
$$
\frac{x^{-1}}{x^{-1}-x^{-2}} = \frac{x}{x-1} < 1+\ep
$$ 
for $x>A$. Then, using (\ref{norm}) we see that if $y<u-A$, then
$$
\frac{P(B(1)> u+a_kb-y)}{P(B(1)> u+b-y)}\le (1+\ep)\exp\left(-\frac{(u+a_kb-y)^2}2 + \frac{(u+b-y)^2}2\right)
$$
$$
=(1+\ep)\exp\left(b(1-a_k)(u-y)+ \frac{b^2(1-a_k^2)}2 \right) 
$$
$$
\le (1+\ep)\exp\left(b(1-a_k)a\log k+ \frac{b^2(1-a_k^2)}2 \right)\,, 
$$
because $u-y < a\log k$. Taking into account that $1-a_k \le mk^{-1}\log k$, we conclude that there is 
an index $k_0$ such that 
\beam
\label{gamma.1}
\gamma_k^{(1)} := \sup_{u-a\log k<y<u-A}\frac{P(B(1)> u+a_kb-y)}{P(B(1)> u+b-y)} \le (1+\ep)^2 \,. 
\eeam

\noindent{\it Step 2.} 
If $y>u-A$, then 
$$
\frac{P(B(1)> u+a_kb-y)}{P(B(1)> u+b-y)} = 1 + 
\frac{\int_{u+a_kb-y}^{u+b-y}e^{-x^2/2}dx }{\int_{u+b-y}^\infty e^{-x^2/2}dx}  
\le \frac{b(1-a_k)}{\int_{A+b}^\infty e^{-x^2/2}dx }\,.
$$ 
So, we can find $k_1$ such that for choosen $A$ the last expression is less that $1+\ep$
for $k>k_1$. Therefore, for given $\ep>0$ there is an index $k_2$ such that
$\gamma_k(u) <(1+\ep)^2$ for all $k>k_2$ and $u>0$.  
 
Denote by $\widetilde R_1(u)$ the sum of summands from (\ref{rep.R1}) over $k>k_2$.  Then
\beam
\label{R1.t}
\limsup_{u\to\infty}\frac{\widetilde R_1(u)}{P(B(1)+Z(1)>u+b)} \le (1+\ep)^2.
\eeam
We have 
$$
 R_1(u) - \widetilde R_1(u) \le e^{-\la}\sum_{k=1}^{k_2}\frac{\la^k}{k!}P(B(1)+S_k>u)\,.
$$ 
On the other hand, for each $j$ 
 $$
 P(B(1)+Z(1)> u+b) > e^{-\la}\frac{\la^{k_2+j+1}}{(k_2+j+1)!}P(B(1)+S_{k_2+j+1}>u+b)
 $$
 $$
 \ge e^{-\la}\frac{\la^{k_2+j+1}}{(k_2+j+1)!}P(B(1)+S_{k_2+1}>u)P(S_j>b). 
 $$
 Choosing $j$ under the condition $P(S_j>b)>0$ and applying Lemma \ref{l:Y}
 we conclude that 
$$
\lim_{u\to\infty}\frac{R_1(u) - \widetilde R_1(u) }{P(B(1)+Z(1)>u+b)} = 0.
$$
Now (\ref{est.R1}) follows  from here and (\ref{R1.t}). 

\noindent{\it Estimate for $R_2(u)$.} We show here that
\beam
\label{est.R2}
\lim_{u\to\infty}\frac{R_2(u)}{P(B(1)+Z(1)>u+b)}=0.
\eeam

The probability $R_2(u)$ admittes a representation similar to (\ref{rep.R1}). Denoting by $g_k(u)$
the corresponding summands we get
$$
g_k(u) \le e^{-\la}\frac{\la^k}{k!}P(B(1)+S_k>u\,,\, S_k\le u-a\log k)\,.
$$
Hence
$$
\delta_k(u) := \frac{g_k(u)}{e^{-\la}\frac{\la^{k+m}}{(k+m)!}P(B(1)+S_{k+m}>u+b)}
$$
$$
\le \frac{(k+1)\cdots(k+m)}{\la^m}\frac{P(B(1)+S_k>u\,,\, S_k\le u-a\log u)}
{P(B(1)+S_k>u-1)P(S_m>b+1)}\,,
$$
where $m$ is choosen under the condition $P(S_m>b+1)>0$.
Once again representing the probabilities as integrals with 
respect to $F_{S_k}$ and using (\ref{norm}) we can find an index $k_1$ such that  
$$
\frac{P(B(1)+S_k>u\,,\, S_k\le u-a\log u)}{P(B(1)+S_k>u-1)}
\le 2\sup_{y\le u-a\log k} \exp\left(-(u-y) +\frac14 \right)
$$
$$
\le 2e^{\frac14}\exp(-a\log k) = 2e^{\frac14}k^{-a}\,
$$ 
for $k>k_1$ and $u>0$, which yields $\delta_k(u) \le Ck^{m-a}$. So,
$$
\frac{\sum_{k=k_k+1}^\infty g_k(u) }{P(B(1)+Z(1)>u+b)} \le Ck_1^{m-a}. 
$$
As in previous case,
$$
\lim_{u\to\infty}\frac{\sum_{k=1}^{k_1}g_k(u) }{P(B(1)+Z(1)>u+b)} = 0\,.
$$
Therefore,
$$
\limsup_{u\to\infty}\frac{R_2(u)}{P(B(1)+Z(1)>u+b)}\le Ck_1^{m-a}\,.
$$
Since $m$ does not depend on $a$, we may choose $a>m$. Then letting $k_1\to\infty$ we obtain (\ref{est.R2}).
Now (\ref{est.R}) follows from (\ref{def.R12}), (\ref{est.R1}) and (\ref{est.R2}). 

According (\ref{Q.til})--(\ref{est.R})
$$
\limsup_{u\to\infty}\frac{P(B(1)+Z(1)>u+b\gt)}{P(B(1)+Z(1)>u+b)}\le 1\,.
$$
Since $P(B(1)+Z(1)>u+b\gt)\ge P(B(1)+Z(1)>u+b)$, (\ref{lim.P}) follows. \qed
\subsection{Proof of (\ref{est.A}).} Because
\label{ss:A}
$$
\sup_{0\le t< \gt} X(t) \le \sup_{0\le t\le 1}B(t) +\sup_{0\le t<\gt}\left[Z(t)-bt \right]\,,
$$
we get, once again applying  L\'evy formula (\ref{sup.B}) and Lemma \ref{l:sym},    
 $$
 A(u) \le 2P\left(B(1) + \sup_{0\le t<\gt}\left[Z(t)-bt \right]>u\right)\,.
 $$
It was shown in \cite{bib:4}, pp. 149--151, that
$$
P\left( \sup_{0\le t<\gt}\left[Z(t)-bt \right]>u\right) = o\left(P(Z(\gt)-b\gt >u) \right)
$$ 
as $u\to\infty$. Hence, according to Lemma \ref{l:comp} 
$A(u) = o\left(P(B(1)+ Z(\gt)-b\gt >u) \right)$. Since $Z(\gt) = Z(1)$,  
(\ref{lim.P}) yields now (\ref{est.A}). \qed

\section{Proof of Theorem \ref{t:pos}.}

As above, (\ref{rho.2}) and Lemma \ref{l:gen} allow us to prove the theorem
for processes of the type (\ref{def.X}). Relation (\ref{maineq}) holds for compound Poisson processes 
with non-negative drifts and light tails (see Theorem 1 from \cite{bib:4}). Hence we may assume that $b>0$.  
 The proof is divided into  a series of lemmas.

\begin{lemma}
\label{l:G}
Assume (\ref{cond.pl}) holds for iid random variables $X_k$ and put
\beam
\label{def.G}
G(u) = \sum_{k=2}^\infty\frac{\la^k P(S_{k-1}>u)}{k!}\,.
\eeam
Let 
\beam
\label{comp.Poi}
Z= \sum_{k=1}^NX_k\,.
\eeam
where $N$ is a Poisson random variable with parameter $\la$, independent
of $X_k$. Then 
$$
\lim_{u\to\infty}\frac{G(u)}{P(Z>u+b)}=0
$$
for each $b>0$.
\end{lemma}

\begin{proof}
Because $P(X_1+X_2>u+b) \ge P(X_1>u-a)P(X_2>a+b)$, relation (\ref{cond.pl}) implies
\beam
\label{S2.b}
\lim_{u\to\infty}\frac{P(X_1>u)}{P(X_1+X_2>u+b)}=0\,.
\eeam
Fix a constant $A>0$. Then
\beam
G(u) \le \la P(X_1>u) + \sum_{k=2}^\infty\frac{\la^k P(S_{k-1}>u\,,\,S_{k-2}\le u-A)}{k!}
\eeam
$$
+ \sum_{k=2}^\infty\frac{\la^k P(S_{k-2}>u-A)}{k!} := \la P(X_1>u) +G_1(u) +G_2(u)\,.
$$
Further, 
\label{G.first}
$$
h(A,k,u) := \frac{P(S_{k-1}>u\,,\,S_{k-2}\le u-A)}{P(S_k>u+b)} 
$$
$$
\le \frac{\int_{-\infty}^{u-A}P(X_1>u-y)F_{S_{k-2}}(dy)}{\int_{-\infty}^{u-A}P(X_1+X_2>u-y+b)F_{S_{k-2}}(dy)}
\le \sup_{y\le u-A}\frac{P(X_1>u-y)}{P(X_1+X_2>u-y+b)}\,,
$$
and it follows from (\ref{S2.b}), that for a fixed $\ep>0$ one can find $A$ such that $h(A,k,u) <\ep$
for all $k\ge 3$ and positive $u$. This yields the estimate $G_1(u) < \ep P(Z>u+b)\,,\,u>0$. 

Turn now to $G_2(u)$. We have for a fixed index $M>2$ 
$$
G_2(u) \le \frac{1}{P(X_1>A+b)}\left( \sum_{k=2}^M \frac{\la^k P(S_{k-1}>u+b)}{k!}
+ \sum_{k=M+1}^\infty\frac{\la^k P(S_{k-1}>u+b)}{k!} \right)\,,
$$
and (\ref{Sk.0}) implies that the first sum is $o(P(Z>u+b))$ as $u\to\infty$. The second sum is
bounded from above by
$$
\frac{1}{M+1}\sum_{k=M+1}^\infty\frac{\la^k P(S_{k-1}>u+b)}{(k-1)!} \le \frac{\la e^\la}{M+1}P(Z>u+b)\,.
$$
So, letting first $u\to\infty$ and then $M\to\infty$ we conclude that $G_2(u) = o(P(Z>u+b))$ as $u\to\infty$
for each $A>0$.
From here and the previous
$$
\limsup_{u\to\infty}\frac{G(u)}{P(Z>u+b)} \le \ep
$$
for each $\ep>0$, which yields the lemma. 
\end{proof}

Because the function $h$ is increasing and continuous, there exists the inverse function $h^{-1}$.
Put for $s>1$
\beam
\label{def.psi}
\psi(s) = h^{-1}\left(\frac4b\log s \right)
\eeam
and
\beam
\label{def.g}
g(k,u) = u - \psi(k)\,.
\eeam

\begin{lemma}
\label{l:gamma}
There is an index $k_0$ such that
$$
\gamma(k,u):= \frac{\frac{\la^k}{k!}P(S_k>u+ba_k\,,\, S_{k-1}\le g(k,u)) }{\frac{\la^{k+1}}{(k+1)!}P(S_{k+1}>u+b)}
\le \frac{C}{\la k}
$$
for all $k>k_0$ and $u>0$, where the constant $C$ is independent  of these parameters.
\end{lemma}
\begin{proof}
We have
$$
\gamma(k,u) \le \frac{k+1}{\la}\frac{P(S_k>u+ba_k\,,\, S_{k-1}\le g(k,u))}{P(S_k>u\,,\,X_{k+1}>b)} 
$$
$$
\le \frac{k+1}{\la P(X_1>b)}
\frac{\int_{-\infty}^{g(k,u)}P(X_1>u+ba_k-y)F_{S_{k-1}}(dy)}{\int_{-\infty}^{g(k,u)}P(X_1>u-y)F_{S_{k-1}}(dy)}
$$
$$
\le \frac{k+1}{\la P(X_1>b)}\sup_{y\le g(k,u)}\frac{P(X_1>u+ba_k-y)}{P(X_1>u-y)}\,.
$$
Because
\beam
\label{u.y}
u-y \ge u-g(k,u) = \psi(k) >0\,,
\eeam
we may apply (\ref{tail.X1}). Hence for $y\le g(k,u)$
$$
 \frac{P(X_1>u+ba_k-y)}{P(X_1>u-y)}= \exp\left(-\int_{u-y}^{u+ba_k-y}h(v)dv \right) \le \exp\left(-ba_kh(u-y) \right)\,.
$$
Taking into account (\ref{u.y}) and (\ref{def.psi}) we see that $h(u-y)\ge h(\psi(k)) =4\log k/b$. 
Formula (\ref{tail.X1}) implies $m=1$ in (\ref{def.m}). So
$ a_k = 1 - \frac{2\log k}k > \frac12$
for $k$ large enough, and, therefore, $ba_kh(u-y) \ge 2\log k\,$.
From here and the previous estimates
$$
\gamma(k,u) \le \frac{k+1}{\la P(X_1>b)}\frac1{k^2}
$$
for $k$ large enough, and the lemma follows.
\end{proof}

Denote
\beam
\label{def.I}
I(k,u)= \la \int_{a_k}^1\frac{(\la t)^{k-1}}{(k-1)!}P(S_k>u+bt\,,\,S_{k-1}\le g(k,u))dt\,.
\eeam
Since
$$
I(k,u) \le \frac{\la^k}{k!}P(S_k>u+ba_k\,,\, S_{k-1}\le g(k,u))\,,
$$
we immediately obtain the following statement.
\begin{corollary}
\label{c:gamma} There is an index $k_0$ such that
$$
\frac{I(k,u)}{\frac{\la^{k+1}}{(k+1)!}P(S_{k+1}>u+b)}
\le \frac{C}{\la k}
$$
for all $k>k_0$ and $u>0$, where the constant $C$ is  independent of these parameters.
\end{corollary}

Denote
\beam
\label{def.J}
J(k,u) = \la \int_{a_k}^1\frac{(\la t)^{k-1}}{(k-1)!}P(S_k>u+bt\,,\,g(k,u)< S_{k-1}\le u)dt\,.
\eeam

\begin{lemma}
\label{l:J}
For each positive $\ep$ there is an index $k_1$ such that
$$
\beta(k,u):= \frac{J(k,u)}{\frac{\la^k}{k!}P(S_k>u+b)} \le 1+\ep
$$
for all $k>k_1$ and $u>0$.
\end{lemma}

\begin{proof}
We have 
\beam
\label{def.alpha}
\beta(k,u)\le \frac{\la\int_{g(k,u)}^u\int_{a_k}^1\frac{(\la t)^{k-1}}{(k-1)!}P(X_1>u+bt-y)dtF_{S_{k-1}}(dy)}
{\frac{\la^k}{k!}\int_{g(k,u)}^u P(X_1>u+b-y)F_{S_{k-1}}(dy)}
\eeam
$$
\le \sup_{g(k,u)<y\le u}\frac{k\int_{a_k}^1P(X_1>u+bt-y)t^{k-1}dt}{P(X_1>u+b-y)} :=\alpha(k,u)\,.
$$

 It follows from (\ref{tail.X1}) that
\beam
\label{def.nu}
\nu(k,u,y) := \frac{k\int_{a_k}^1P(X_1>u+bt-y)t^{k-1}dt}{P(X_1>u+b-y)}
\eeam
$$
= k\int_{a_k}^1\exp\left(\int_{u+bt-y}^{u+b-y}h(v)dv \right) t^{k-1}dt 
= k\int_{a_k}^1\exp\left(b(1-t)h(v(u,y,t) \right) t^{k-1}dt\,, 
$$
where 
$$
u+bt-y < v(u,y,t) <u+b-y\,.
$$
Because $u+b-y <u+b-g(k,u) = \psi(k)+b$, we get, using (\ref{cond.h}), 
$$
h(v(u,y,t)) \le h(\psi(k)+b) = h\left(h^{-1}\left(\frac{4\log k}b \right) +b \right)
\le \exp\left(\frac b8\frac{4\log k}b  \right) = \sqrt k
$$
for $k$ large enough. Since $a_k\le t <1$, we see, taking into account (\ref{def.ak})  that
$$
b(1-t)h(v(u,y,t)) \le \frac{2b\log k}{k}\sqrt k\ \to 0
$$
 as $k\to\infty$. So, (\ref{def.nu}) and the last estimates imply that there is an index $k^\prime$ such that
$\nu(k,u,y) < 1+\ep$ for all $k>k^\prime\,,\, u>0$ and $g(k,u)<y<u$, which yields $\alpha(k,u) <1+\ep$
for $k>k^\prime$ and $u>0$, and the lemma follows.

\end{proof}

\noindent{\it Proof of Theorem \ref{t:pos}.} According to Theorem 1 from \cite{bib:4} it is enough to show that
\beam
\label{H.eq}
\lim_{u\to\infty}\frac{P(Z(1)> u+b\gt )}{P(X(1)>u)} =1\,.
\eeam
Applying (\ref{def.Q}), (\ref{def.G}), (\ref{def.I}) and (\ref{def.J}) we may write
\beam
\label{H.rep}
P(Z(1)> u+b\gt ) \le Q(u) +  e^{-\la}\sum_{k=1}^\infty I(k,u) + e^{-\la}\sum_{k=1}^\infty J(k,u) +e^{-\la}G(u) \,.
\eeam
We show first that
\beam
\label{I.0}
\lim_{u\to\infty}\frac{\sum_{k=1}^\infty I(k,u)}{P(X(1)>u)} = 0\,.
\eeam
 Since $I(k,u) \le \la^k P(S_k>u)/k!$,
relations (\ref{tail.cP})  and (\ref{Sk.0}) yield
$$
\lim_{u\to\infty}\frac{\sum_{k=1}^M I(k,u)}{P(X(1)>u)} = 0
$$
for each fixed index $M$. Choosing $M>k_0$, where $k_0$ is from Corollary \ref{c:gamma}, we 
conclude that
$$
\frac{\sum_{k=M+1}^\infty I(k,u)}{P(X(1)>u)} \le \frac{C}{\la (M+1)} 
$$
for all $u>0$. Letting first $u\to\infty$ and then  $M\to\infty$ we come to (\ref{I.0}).

Now we show that
\beam
\label{J.1}
\limsup_{u\to\infty}\frac{e^{-\la}\sum_{k=1}^\infty J(k,u)}{P(X(1)>u)} \le 1. 
\eeam
For a fixed $\ep>0$   Lemma \ref{l:J} and (\ref{tail.cP}) provide us with an index $k_1$ such that
$$
\frac{e^{-\la}\sum_{k=k_1+1}^\infty J(k,u)}{P(X(1)>u)} < 1+\ep
$$
for all $u>0$. As above, (\ref{tail.cP}) and (\ref{Sk.0}) lead to the equality
$$
\lim_{u\to\infty}\frac{\sum_{k=1}^{k_1} J(k,u)}{P(X(1)>u)} = 0. 
$$
The last two relations imply (\ref{J.1}).

Now (\ref{H.rep}), (\ref{I.0}) , (\ref{J.1}) and Lemmas \ref{l:int} and \ref{l:G} give us
\beam
\label{limsup.H}
\limsup_{u\to\infty}\frac{P(Z(1)> u+b\gt )}{P(X(1)>u)} \le 1\,.
\eeam
Obviously, $P(Z(1)> u+b\gt )\ge P(X(1)>u)$ for all $u>0$, and we come to (\ref{H.eq}). \qed

\section{Proof of Theorem \ref{t:neg}.}

First we prove the following

\begin{proposition}   
\label{p:pl}
If  for $X_k$ the condition (\ref{cond.pl}) holds, then it also holds for random variable (\ref{comp.Poi}). 
\end{proposition}

The proof is based on the next statement.

\begin{lemma}
\label{l:pl}
Assume $X_k$ satisfy (\ref{cond.pl}). Then for each  $\ep>0$ there is $B>0$ such that
$$
P(S_k>u+a)\le \ep P(S_k>u) + \frac{P(S_{k-1}>u)}{P(X_1>B)}
$$
for all $k\ge 3$ and $u>0$.
\end{lemma}
\begin{proof}
Fix $A>0$. Then 
$$
P(S_k>u+a) \le \int_{-\infty}^{u-A}P(X_1>u+a-t)F_{S_{k-1}}(dt) + P(S_{k-1}>u-A)\,.
$$
Because of (\ref{cond.pl}), there is $A_0>0$ such that
$
{P(X_1>u+a-t)}/{P(X_1>u-t)} < {\ep}/2
$
for all $A>A_0$ and $t<u-A$.
Then 
\beam
\label{ineq.1}
P(S_k>u+a)\le \frac{\ep}2 P(S_k>u) + P(S_{k-1}>u-A)\,.
\eeam

Fix now $A>A_0$. We have for a positive $B$:
\beam
\label{ineq.2}
P(S_{k-1}>u-A) \le \int_{-\infty}^{u-B}P(X_1>u-A-t)F_{S_{k-2}}(dt) + P(S_{k-2}>u-B)\,.
\eeam
Further, $P(X_1+X_2>u-t) \ge P(X_1>u-t-2A)P(X_1>2A)$,
and $t\le u-B$ implies $u-A-t\ge B-A$. Hence, once again applying (\ref{cond.pl}), we can choose $B$ so large that
$
{P(X_1>u-A-t)}/{P(X_1+X_2>u-t)} < {\ep}/2
$
if $t\le u-B$. Therefore,
$$
\int_{-\infty}^{u-B}P(X_1>u-A-t)F_{S_{k-2}}(dt) \le \frac{\ep}2\int_{-\infty}^{\infty}P(X_1+X_2>u-t)F_{S_{k-2}}(dt)
= \frac{\ep}2P(S_k>u)\,.
$$
Since
$
P(S_{k-2}>u-B) \le {P(S_{k-1}>u)}/{P(X_1>B)}\,,
$
the lemma follows from here, (\ref{ineq.2}) and (\ref{ineq.1}).
\end{proof}



\noindent{\it Proof of Proposition \ref{p:pl}.} According to Lemma \ref{l:pl}, for a fixed $\ep>0$ 
there is $B>0$ such that
\beam
\label{Z.pl}
P(Z>u+a) \le e^{-\la}\left[\la P(S_1>u+a) +\frac{\la^2P(S_2>u+a)}{2!} \right]
\eeam
$$
+ \ep e^{-\la}\sum_{k=3}^\infty\frac{\la^k P(S_k>u)}{k!} 
+ \frac{e^{-\la}}{P(X_1>B)}\sum_{k=3}^\infty\frac{\la^k P(S_{k-1}>u)}{k!}\,. 
$$
We show first the the last sum is $o(P(Z>u))$ as $u\to\infty$. To this end fix an index $m>3$.
Then
$$
\psi(u) := \sum_{k=3}^\infty\frac{\la^k P(S_{k-1}>u)}{k!} \le 
\sum_{k=3}^m\frac{\la^k P(S_{k-1}>u)}{k!} 
+ \frac{1}{m+1}\sum_{k=m+1}^\infty \frac{m+1}{k}\frac{\la^k P(S_{k-1}>u)}{(k-1)!}\,.
$$
Taking into account (\ref{Sk.0}) we see that
$$
\limsup_{u\to\infty}\frac{\psi(u)}{P(Z>u)} \le \frac{\la e^\la}{m+1}\,,
$$
and letting $m\to\infty$ we come to the needed conclusion.

Now, (\ref{Z.pl}) and (\ref{Sk.0}) yield that
$$
\limsup_{u\to\infty}\frac{P(Z>u+a)}{P(Z>u)} \le \ep
$$
for each $\ep>0$. So, the proposition follows. \qed

\noindent{\it Proof of Theorem \ref{t:neg}.} The proof  is a word for word repetition of the proof of Theorem 2
from \cite{bib:4}. To obtain formula (32) from  this paper one should use Proposition \ref{p:pl}. 

\section{Proof of Theorem \ref{t:example} }
\label{s:proof2}

Let $\{X_k\}_{k=1}^\infty$ be iid random variables taking values $n!\,, n=1,2,\dots$, such that
\beam
\label{n!}
P(X_1=n!)=\frac{1}{(e-1)n!}\,.
\eeam
Denote by $Z(t)$ a compound Poisson process with parameter $\la=1$ and jumps $X_k$. It will be shown below that
(\ref{limsup}) holds for the sequence $u_n = n!$ and (\ref{liminf}) holds for the sequence $u_n = n\cdot n!$.

\subsection {Estimates for sums} Here we obtain asymptotics for probabilities $P(S_k+B(1)>n!)$
and $P(S_k+B(1)>n\cdot n!)$ as $n\to\infty$. 
\begin{lemma}
\label{l:as1}
For $2\le k \le n$
$$
P(S_k+B(1)>n!) = kP(X_1=n!)\int_{-\infty}^\infty P(S_{k-1}>-t)\phi(t)dt + \frac{\alpha(k,n)}{(n+1)!}\,,
$$
where $\phi$ is $(0,1)$-normal density function and
$
\sup_{2\le k\le n}|\alpha(k,n)| <\infty\,.
$
\end{lemma}
\begin{proof} 
We represent the considered probability as
\beam
\label{rep.int}
P(S_k+B(1)>n!) = \left(\int_{-\infty}^0 + \int_0^n + \int_n^\infty \right)P(S_k>n!-t)\phi(t)dt
:= I_1 +I_2+I_3\,,
\eeam
and start with the integral $I_1$. 
Assume first that $k=n$. The condition  
$$\max\{X_1,\dots,X_n\}\le (n-1)!
$$ 
implies $S_n\le n!-t$ for $t\le 0$. Hence,
$$
P(S_n>n!-t) = P\left(S_n>n!-t\,,\, \max\{X_1,\dots,X_n\} \ge n!\right)
$$
$$
= P\left(S_n>n!-t\,,\, \mbox{exactly one of $X_1,\dots,X_n$ is non-less than $n!$}\right)
$$
$$
+P\left(S_n>n!-t\,,\, \mbox{at least two of $X_1,\dots,X_n$ are non-less than $n!$}\right) : =p + q.
$$
Since $X_k$ are iid random variables,
$$
p = nP\left(S_n>n!-t\,,\,X_1\ge n!\,, \max\{X_2,\dots,X_n\} \le (n-1)!\right)
$$
$$
= nP\left(S_n>n!-t\,,\,X_1= n!\right)
- nP\left(S_n>n!-t\,,\,X_1= n!\,, \max\{X_2,\dots,X_n\} \ge n!\right)
$$
$$
+ nP\left(S_n>n!-t\,,\,X_1\ge (n+1)!\,, \max\{X_2,\dots,X_n\} \le (n-1)!\right)
:= n(p_1-p_2+p_3)\,.
$$
Obviously,
$
p_1 = P(X_1=n!)P(S_{n-1}>-t)\,,
$
and according to (\ref{n!}),
$$
p_2 \le P(X_1=n!)P\left( \max\{X_2,\dots,X_n\} \ge n!\right) \le C\frac{1}{n!}\frac{n-1}{n!}
$$
and
$$
p_3 \le P(X_1\ge (n+1)!) \le \frac{C_1}{(n+1)!}\,,
$$
 It is clear that
$$
q \le n^2\left[P(X_1\ge n!) \right]^2 \le \frac{C_2}{[(n-1)!]^2}\,.
$$ 
From here
\beam
\label{I1}
I_1 = nP(X_1=n!)\int_{-\infty}^0 P(S_{n-1}>-t)\phi(t)dt + O\left(\frac1{(n+1)!} \right)\,. 
\eeam

 Turn now to the integral $I_2$. For  $0< t < n$
$$
P(S_n>n!-t) = P\left(S_n>n!-t\,,\, \max\{X_1,\dots,X_n\} \ge n!\right)
$$ 
$$
+P\left(S_n>n!-t\,,\, \max\{X_1,\dots,X_n\} \le (n-1)!\right) := \widetilde p +  \widetilde q\,,
$$
and by the same reasons as above
$
\widetilde p = nP(X_1=n!)P(S_{n-1}>-t) + O\left({1}/{(n+1)!} \right)\,.
$
Further,
$$
\widetilde q= P(S_n >n!-t\,, X_1=\cdots =X_n= (n-1)!)
$$
$$
+P\left(S_n >n!-t\,, \max\{X_1,\dots,X_n\} =(n-1)!\,,\,\min\{X_1,\dots,X_n\} \le (n-2)! \right)
$$
$$
+P\left(S_n >n!-t\,, \max\{X_1,\dots,X_n\} \le (n-2)! \right) :=\widetilde q_1 +\widetilde q_2
+\widetilde q_3\,.
$$
If $\max\{X_1,\dots,X_n\} =(n-1)!$ and $\min\{X_1,\dots,X_n\} \le (n-2)!$, then
$$
S_n \le (n-1)\cdot(n-1)! + (n-2)! = n! - (n-2)\dot(n-2)! < n! -t
$$
for $t< n$. So, $\widetilde q_2=0$. By similar reasons $\widetilde q_3=0$ and
$$
\widetilde q_1 \le \left[P(X_1=(n-1)!) \right]^n\,. 
$$ 
Hence
$$
I_2 = nP(X_1=n!)\int_0^n P(S_{n-1}>-t)\phi(t)dt + O\left(\frac{1}{(n+1)!}\right)\,.
$$ 
Since 
\beam
\label{I3}
I_3 = O\left(\exp\left(-\frac{n^2}{2}\right) \right) = o \left(\frac{1}{(n+2)!} \right)
\eeam
the last relations and (\ref{rep.int}) yield the lemma for $k=n$. 

The case $2\le k <n$ is treated by the similar way.  
 \end{proof}
 
\begin{remark}{\rm
 \label{k=1}
The same reasons give us
\beam
\label{X1}
P(S_1+B(1)>n!) = O\left(\frac1{(n+1)!} \right)\,.
\eeam}
\end{remark}

\begin{lemma}
\label{l:as2}
For $2\le k\le n$
$$
P(S_k+B(1)>n\cdot n!) = kP(X_1=(n+1)!) + \frac{\beta(k,n)}{(n+2)!}\,,
$$
where
$
\sup_{2\le k\le n}|\beta(k,n)| <\infty\,.
$
\end{lemma} 
\begin{proof} Since the conditions $\max\{X_1,\dots,X_n\} =n!$ and $\min\{X_1,\dots,X_n\} \le (n-1)!$
imply
$
S_n \le (n-1)n! + (n-1)! = n\cdot n! - (n-1)(n-1)! < n\cdot n! -n\,,
$
then for $t<n$
$$
P(S_n > n\cdot n!-t) = P\left(S_n > n\cdot n!-t\,,\, X_1=\cdots =X_n=n! \right)
$$
$$
+ P\left(S_n > n\cdot n!-t\,,\,\max\{ X_1,\dots , X_n\}\ge (n+1)! \right)\,. 
$$
From here, as in the proof of the previous lemma,
$$
P(S_n > n\cdot n!-t) = nP(X_1=(n+1)!)P(S_{n-1}> -n!-t)+O\left(\frac1{(n+2)!} \right)\,,
$$
because $n\cdot n!-(n+1)!= -n!$. In the case $2\le k <n$ we get similarly for $t<n$
\beam
\label{Sk.nnt}
P(S_k > n\cdot n!-t) = kP(X_1=(n+1)!)P(S_{k-1}> -n!-t)+ \frac{\mu(k,n,t)}{(n+2)!}\,,
\eeam
where
$
\sup_{2\le k<n\,;t<n}|\mu(k,n,t)| <\infty\,.
$
These equalities and (\ref{I3}) give us
\beam
\label{Sk.nn}
P(S_k+B(1)>n\cdot n!) = kP(X_1=(n+1)!)\int_{-\infty}^\infty P(S_{k-1}> -n!-t)\phi(t)dt
\eeam
$$
+\frac{\widetilde\mu(k,n)}{(n+2)!}
$$
and
$ 
\sup_{2\le k\le n}|\widetilde\mu(k,n)|<\infty\,.
$
Since $X_k$ are positive, we have for $t>-n!$
\beam
\label{Sk.lower}
P(S_{k-1}> -n!-t) = 1\,.
\eeam
Hence, the last integral is
$$
\int_{-\infty}^{-n!}P(S_{k-1}> -n!-t)\phi(t)dt + P(B(1)>-n!)= 1 + o\left(\frac1{(n+3)!} \right)\,.
$$
From here and the previous relations the lemma follows.
\end{proof}

\begin{remark}{\rm
\label{r:k=12}
By the same way we obtain
\beam
\label{X12}
P(S_1+B(1)>n\cdot n!) = 
 P(X_1=(n+1)!)P(B(1)>-n!)+ O\left(\frac1{(n+2)!} \right) 
\eeam
$$
= P(X_1=(n+1)!)+ O\left(\frac1{(n+2)!} \right) 
$$}
\end{remark}
\subsection{Estimates for $X(1)$} Here we find asymptotics for the probabilities $P(X(1)> n!)$ and $P(X(1)> n\cdot n!)$.
\begin{lemma}
\label{l:est.X}
The following hold:
\beam
\label{X.n!}
P(X(1)>n!) = P(X_1=n!)\sum_{k=2}^nI_k + O\left(\frac1{(n+1)!} \right)\,, 
\eeam
where
\beam
\label{Ik}
I_k= \frac1{e(k-1)!}\int_{-\infty}^\infty P(S_{k-1}>-t)\phi(t)dt\,.
\eeam
\end{lemma}

\begin{proof}
We can write the considered probability as a sum of three sums:
$$
P(X(1)>n!) = e^{-1}\left[P(B(1)>n!) + P(S_1+B(1))>n!\right] 
$$
$$
+ e^{-1}\sum_{k=2}^n\frac{P(S_k+B(1)>n!)}{k!} +e^{-1}\sum_{n+1}^\infty\frac{P(S_k+B(1)>n!)}{k!},
$$
and (\ref{X1}) implies that the first sum is $O\left(1/{(n+1)!}\right)$. The same is true for the third sum.
As for the second one, Lemma \ref{l:as1} yields that it is
$$
P(X_1=n!)\sum_{k=2}^n\frac1{e(k-1)!}\int_{-\infty}^\infty P(S_{k-1}>-t)\phi(t)dt + O\left(\frac1{(n+1)!}\right)\,
$$
and (\ref{X.n!}) follows.
\end{proof}

\begin{lemma}
\label{l:est.X1}
The following relation holds:
\beam
\label{X.nn!}
P(X(1)>n\cdot n!) = P(X_1= (n+1)!) + O\left(\frac1{(n+2)!} \right)\,.
\eeam
\end{lemma}
\begin{proof}
 We can write, using (\ref{X12}) and Lemma \ref{l:as2}, 
$$
P(X(1)>n\cdot n!) = e^{-1}P(B(1)>n\cdot n!) + P(X_1= (n+1)!)e^{-1}\sum_{k=1}^n\frac1{(k-1)!}
$$
$$
+ e^{-1}\frac{P(S_{n+1}>n\cdot n!)}{(n+1)!} + O\left(\frac1{(n+2)!} \right)\,.
$$
Since the condition $\max\{X_1,\dots,X_{n+1}\}\le (n-2)!$ implies $S_{n+1}\le (n+1)(n-2)! < n\cdot n!$,
we see that 
\beam
\label{est.Sn+1}
P(S_{n+1}>n\cdot n!) = P\left(S_{n+1}>n\cdot n!\,,\,\max\{X_1,\dots,X_{n+1}\}\ge (n-1)! \right)
\le \frac{C(n+1)}{(n-1)!}\,.
\eeam
 So, the needed relation follows. 
\end{proof}

\subsection {Proof of (\ref{limsup})} We have
\beam
\label{sup.1}
P\left(\sup_{0\le t \le 1}X(t) > n! \right) \ge P(X(1)>n!) + P(X(1)\le n!\,,\, X(\Gamma_{\tau-1})>n!)\,,
\eeam
where $\tau$ is given by (\ref{def.tau}). Further,
$$
p_k(n!) := P(\tau =k\,,\,X(1)\le n!\,,\, X(\Gamma_{k-1})>n!) 
$$
$$
= P\left(\tau =k\,,\,S_k+B(1)\le n!\,,\, S_{k-1}+B(\Gamma_{k-1})>n! \right) 
$$
Because $B(t)$ is symmetric and independent of $Z(t)$,
\beam
\label{p.low}
p_k(n!) \ge \frac12 P\left(\tau =k\,,\,S_k+B(\Gamma_{k-1})\le n!\,,\, S_{k-1}+B(\Gamma_{k-1})>n! \right)
:=\frac12 q_k(n!)\,. 
\eeam
Elementary calculations give us 
$$
q_k(n!) = \frac{1}{e(k-2)!}\int_0^1P\left(S_k+B(y)\le n!\,,\, S_{k-1}+B(y)>n! \right)
y^{k-1}(1-y)dy\,.
$$

Assume now that $3\le k\le n$. The same reasons as above and  the well known formula for the density of $B(y)$
imply
\beam
\label{rep.q}
q_k(n!) = (k-1)P(X_1=n!)
\eeam
$$
\times\frac{1}{e(k-2)!}\int_0^1\left[\int_{-\infty}^\infty P(S_{k-1}\le -t\,,\,S_{k-2}>-t)\frac1{\sqrt{2\pi y}}e^{-\frac{t^2}{2y}}dt\right]
y^{k-1}(1-y)dy 
$$
$$
+ \frac{\nu(k,n)}{(n+1)!} := (k-1)P(X_1=n!)J_{k} +  \frac{\nu(k,n)}{(n+1)!}\,,
$$
where 
$
\sup_{3\le k\le n}|\nu(k,n)| < \infty\,.
$
Because the jumps $X_k$ are positive, the inner integral coinsides with the integral over $(-\infty, -1)$,
and it is positive. So, 
\beam
\label{J.pos}
J_{k} >0 \quad \mbox{for all $ k\ge 3$}\,
\eeam
and (\ref{sup.1}) and (\ref{p.low}) imply 
\beam
\label{sup.lower}
P\left(\sup_{0\le t \le 1}X(t) > n! \right) \ge P(X(1)>n!) + \frac12 P(X_1=n!)\sum_{k=3}^n(k-1)J_{k}
+O\left(\frac1{(n+1)!} \right)\,. 
\eeam
According to (\ref{Ik}) $\sum_{k=2}^\infty I_k \le 1$.
From here, (\ref{sup.lower}), (\ref{n!}), (\ref{X.n!}) and (\ref{J.pos})
$$
\liminf_{n\to\infty}\frac{P\left(\sup_{0\le t \le 1}X(t) > n! \right)}{P(X(1)>n!)}
\ge 1 + \frac{\frac12 \sum_{k=3}^\infty(k-1)J_k}{\sum_{k=2}^\infty I_k} >1\,,
$$
and (\ref{limsup}) follows. \qed

\subsection{Proof of (\ref{liminf})} Using (\ref{sup.B}) and the positivity of $Z(t)$ we may write
\beam
\label{sup.nn}
P\left(\sup_{0\le t \le 1}X(t) > n\cdot n! \right) \le 
P\left(Z(1)+|B(1)| > n\cdot n! \right)\,.
\eeam
Applying (\ref{Sk.nnt}) and (\ref{Sk.lower}) we see that for $2\le k\le n$ and $0<t<n$ 
$$
P\left(S_k >n\cdot n! -t\right) = kP(X_1=(n+1)!) 
+ \frac{\widetilde\alpha(k,n,t)}{(n+2)!}
$$
and
$
\sup_{2\le k\le n\,;\,0<t<n}|\widetilde\alpha(k,n,t)|<\infty\,.
$
Integrating with respect to the distribution of $|B(1)|$ and using 
 (\ref{I3}) imply for $1\le k\le n$:
\beam
\label{max.SkB}
P\left(S_k +|B(1)| >n\cdot n! \right) = kP(X_1=(n+1)!) + O\left(\frac1{(n+2)!}\right)\,.
\eeam
The same reasons as in the proof of (\ref{est.Sn+1}) yield $P\left(S_{n+1} >n\cdot n! -t\right) \le C/(n-2)!$
for $0<t<n$. Applying (\ref{I3}) we conclude that 
$$
P\left(S_{n+1} +|B(1)| >n\cdot n! \right) = O\left( \frac{1}{(n-2)!} \right)\,.  
$$
From here and (\ref{sup.nn}) 
$$
P\left(\sup_{0\le t \le 1}X(t) > n\cdot n! \right) \le P(X_1=(n+1)!)\sum_{k=1}^n\frac1{e(k-1)!} 
+O\left(\frac1{(n+2)!} \right)\,,
$$
and (\ref{X.nn!})  and (\ref{n!}) imply that
$$
\limsup_{n\to\infty}\frac{P\left(\sup_{0\le t \le 1}X(t) > n\cdot n! \right)}
{P\left(X(1) > n\cdot n! \right)} \le 1\,.
$$
So, (\ref{liminf}) follows. \qed

\begin{remark}{\rm
\label{r:moment}
According to (\ref{n!}), the jumps $X_k$ of compound Poisson process $Z$ have not moments of positive
order. But one can consider jumps with the distribution
$$
 P(X_1=n!) = \frac{C(v)}{(n!)^v}\,,\, n=1,2,\dots,
$$  
where $v$ is a positive constant and $C(v)$ is the corresponding norming constant. Now jumps have finite moments 
of order less than $v$, and almost the same proof gives Theorem \ref{t:example}.  } 
\end{remark}

\section{Some comments}
\label{s:remarks}

\subsection{About the proof of Theorem \ref{t:main}} 
Looking on (\ref{equiv}) one may assume that the relation
$$
\lim_{u\to\infty}\frac{P\left(\sup_{0\le t\le 1}B(t) + \sup_{0\le t\le 1}Z(t)>u\right)}{P(B(1)+Z(1)>u)} =1
$$
also holds, which might shorten the proof of (\ref{est.C}).
It is true if the tail of $X_k$ is subexponential (see, for example, Proposition 2.1 from
\cite{bib:9} and references therein). Here we show that it is not true for 
light tails. 

\begin{proposition}
\label{p:pl2}
If (\ref{cond.pl}) holds for $X_k$, then
$$
\lim_{u\to\infty}\frac{P\left(\sup_{0\le t\le 1}B(t) + \sup_{0\le t\le 1}Z(t)>u\right)}{P(B(1)+Z(1)>u)} =2\,.
$$
\end{proposition}
\begin{proof}
Clearly that
$$
P(B(1)+Z(1)>u) = P(B(1)+Z(1)>u\,,\,B(1)>0) + P(B(1)+Z(1)>u\,,\,B(1)<0)\,, 
$$
and the last probability does not exceed $P(Z(1)>u)$. On the other hand.
$$
P(B(1)+Z(1)>u) \ge P(Z(1)>u-a)P(B(1)>a)
$$ 
for $a>0$. From here and Proposition \ref{p:pl}
$$
\lim_{u\to\infty}\frac{P(B(1)+Z(1)>u\,,\,B(1)<0)}{P(B(1)+Z(1)>u)} =0\,,
$$
and , therefore,
$$
\lim_{u\to\infty}\frac{P(B(1)+Z(1)>u\,,\,B(1)>0)}{P(B(1)+Z(1)>u)} =1\,.
$$
Relation (\ref{sup.B}) implies
$$
P(B(1)+Z(1)>u\,,\,B(1)>0) = \frac12 P(|B(1)|+Z(1)>u)
=\frac12 P\left(\sup_{0\le t\le 1}B(t) +Z(1)>u\right)\,.
$$
But, according to Theorem 1 from \cite{bib:4} 
\beam
\label{sup.l}
P\left(\sup_{0\le t\le 1}Z(t)>u\right)\sim P\left(Z(1)>u\right)
\eeam
as $u\to\infty$. So, Lemma \ref{l:comp} implies
$$
\lim_{u\to\infty}\frac{P\left(|B(1)| + \sup_{0\le t\le 1}Z(t)>u\right)}{P(|B(1)|+Z(1)>u)} =1\,. 
$$
The last equalities yield the proposition.
\end{proof}

The limit considered in Proposition \ref{p:pl2} can belong to the interval $(1,2)$. 
To show this, assume that jumps $X_k$ satisfy the condition
$$
P(X_k>u) \sim e^{-\alpha u}u^\gamma\quad\mbox{as $u\to\infty$}\,,
$$ 
where $\alpha>0$ and $\gamma >-1$ are constants. Then  (\ref{light1}) holds. For a random 
variable $Y$ denote
$$
m_Y^+(\alpha)= \int_0^\infty e^{\alpha t}F_Y(dt)\,,\,m_Y^-(\alpha)= \int_{-\infty}^0 e^{\alpha t}F_Y(dt)\,,
$$ 
and $m_Y(\alpha)= m_Y^+(\alpha) + m_Y^-(\alpha)$. Lemma 4 from \cite{bib:44} gives us
$$
\lim_{u\to\infty}\frac{P(|B(1)| +Z(1)>u)}{P(Z(1)>u)} = m_{|B(1)|}(\alpha) = 2m_{B(1)}^+(\alpha)
$$
and
$$
\lim_{u\to\infty}\frac{P(B(1) +Z(1)>u)}{P(Z(1)>u)} = m_{B(1)}(\alpha)\,. 
$$
As above, applying  (\ref{sup.l}) and Lemma \ref{l:comp} we see that the limit under consideration is
$$
l := \frac{2m_{B(1)}^+(\alpha)}{m_{B(1)}^+(\alpha) + m_{B(1)}^-(\alpha)}\,.
$$
Because $m_{B(1)}^-(\alpha) < m_{B(1)}^+(\alpha)$, we conclude that $1<l<2$.

\subsection{A conjecture}

It is well known that if $X(t)$ is a symmetric L\'evy process, then
\beam
\label{Doob} 
P\left(\sup_{0\le t\le 1}X(t)>u \right) \le 2P(X(1)>u)
\eeam
for all $u>0$. Comparing it with (\ref{B.sup}) and Theorem \ref{t:main} naturally yields the following

{\bf Conjecture.} {\it Let $X(t)$ be a symmetric L\'evy process such that
\beam
\label{conj}
\limsup_{u\to\infty}\frac{P\left(\sup_{0\le t\le 1}X(t)>u \right)}{P(X(1)>u)} = 2\,.
\eeam
Then $X(t) = \sigma B(t)$ for a positive constant $\sigma$.} 

The following  statement supports this conjecture.

\begin{proposition}
\label{p:main} Let $X(t)$ be a symmetric compound Poisson process. Then
the strong inequality holds:
\beam
\label{main}
\limsup_{u\to\infty}\frac{P(\sup_{0\le t\le 1}X(t)>u )}{P(X(1)>u)} <2.
\eeam
\end{proposition}

To show it we need the following 
\begin{lemma}
\label{l:one}
Let $X(t)$ be a symmetric compound Poisson process with jumps $X_k$ and parameter $\la$. Then for all positive $u$:
$$
P\left(\sup_{0\le t \le 1}X(t)>u \right)\le 2P\left(X(1)>u \right) - D(u)\,,
$$
where
\beam
\label{D}
D(u) = e^{-\la}\left[\la P(X_1>u) +\sum_{n=2}^\infty \frac{\la^n}{n!}
P\left(\max_{1\le k \le n-1}S_k\le u,\,S_n>u \right)\right]\,.
\eeam
\end{lemma}
\begin{proof}
The proof is a modification of the proof of Levy inequality (see \cite{bib:55}, p. 50). We have
$$
P\left(\sup_{0\le t\le 1}X(t)> u\right) = P(X(1)>u) + P\left(\sup_{0\le t\le 1}X(t)> u,\, X(1) \le u\right)\,,
$$
and
\beam
\label{main.eq}
P\left(\sup_{0\le t\le 1}X(t)>u,\,X(1)\le u \right)
= e^{-\la}\left[\sum_{n=2}^\infty\frac{\la^n}{n!}P\left(\max_{1\le k\le n-1}S_k >u,\, S_n\le u \right) \right]\,.
\eeam
For any $n> 2$
$$
P\left(\max_{1\le k\le n-1}S_k >u,\, S_n\le u \right) 
$$
$$
= P(S_1>u,\, S_n\le u) + \cdots+
P(S_1\le u,\dots, S_{n-2}\le u,\, S_{n-1}>u, S_n\le u)
$$
$$
\le P(S_1>u,\,X_2+\cdots+X_n\le 0) +\cdots+ P(S_1\le u,\dots, S_{n-2}\le u, S_{n-1}>u,\,X_n\le 0).
$$
Since random variables $X_k$ are independent and symmetric, the last line can be written as
$$
P(S_1>u,\, X_2+\cdots+X_n\ge 0)+\cdots+ P(S_1\le u,\dots, S_{n-2}\le u,\,S_{n-1}>u, X_n\ge 0)
$$
$$
\le P(S_1>u,\,S_n>u)+\cdots+P(S_1\le u,\dots,S_{n-2}\le u,\, S_{n-1}>u,\, S_n>u)
$$
$$
= P\left(\max_{1\le k\le n-1}S_k>u,\, S_n>u \right)
= P(S_n>u) -P\left(\max_{1\le k \le n-1}S_k\le u,\, S_n>u \right)\,.
$$
The same inequality holds for $n=2$. Therefore,
$$
P\left(\sup_{0\le t \le 1}X(t)>u\right) \le P(X(1)>u) +  
e^{-\la}\sum_{n=2}^\infty\frac{\la^n}{n!}P(S_n>u)
$$
$$
- e^{-\la}\sum_{n=2}^\infty\frac{\la^n}{n!}P\left(\max_{1\le k\le n-1}S_k\le u,\,S_n>u \right)
= 2P(X(1)>u) -D(u)\,. 
$$
\end{proof}

{\it Proof of Proposition \ref{p:main}.}
If the upper limit in (\ref{main})
is equal to 2,  there is a sequence $u_j\to\infty$ such that
\beam
\label{uj1}
\lim_{j\to\infty}\frac{P(\sup_{0\le t \le 1}X(t)>u_j)}{P(X(1)>u_j)} =2\,.
\eeam
Then Lemma \ref{l:one} yields
\beam
\label{D.zero}
\lim_{j\to\infty}\frac{D(u_j)}{P(X(1)>u_j)} =0\,,
\eeam
which implies, in particular,
$$
\lim_{j\to\infty}\frac{P(X_1>u_j)}{P(X(1)>u_j)} =0\,.
$$
Applying L\'evy inequality we get for $n\ge 2$
$$
P\left(\max_{1\le k\le n-1}S_k\le u,\, S_n>u \right)
\ge P(S_n>u) -P\left(\max_{1\le k\le n-1} S_k>u\right)
$$
$$
\ge P(S_n>u) -2P(S_{n-1}>u)\,.
$$
Using (\ref{D.zero}), (\ref{D})  and the induction one comes to the relation
\beam
\label{Sn.0}
\lim_{j\to\infty}\frac{P(S_n>u_j)}{P(X(1)>u_j)}= 0
\eeam
for all $n\ge 2$. 

Further, once again using L\'evy inequality we get from (\ref{main.eq})
$$
P\left(\sup_{0\le t\le 1}X(t)>u_j,\,X(1)\le u_j \right) \le 2e^{-\la}\sum_{n=2}^\infty\frac{\la^n}{n!}P(S_{n-1}>u_j),
$$
and the same reasons as in the proof of Proposition \ref{p:pl} show that this sum is 
$o(P(X(1))>u_j)$ as $j\to\infty$. Hence,
$$
\lim_{j\to\infty}\frac{P(\sup_{0\le t \le 1}X(t)>u_j)}{P(X(1)>u_j)} = 1\,,
$$
which contradicts to (\ref{uj1}). \qed

\end{document}